\documentclass{amsart}
\usepackage{amsthm}
\usepackage{amssymb}
\usepackage{amsopn}
\usepackage{epsfig}
\usepackage{amsfonts}
\usepackage{latexsym}
\usepackage{graphicx}
\usepackage{calrsfs}

\newtheorem{theorem}{Theorem}[section] 
\newtheorem{lemma}[theorem]{Lemma}     
\newtheorem{corollary}[theorem]{Corollary}
\newtheorem{proposition}[theorem]{Proposition}
\newtheorem{example}[theorem]{Example}
\newtheorem{remark}[theorem]{Remark}
\newtheorem{definition}[theorem]{Definition}

\newcommand{\Ga}[1]{\Gamma_{{\beta},#1}}
\newcommand{\U}[1]{\mathbf{U}_{{\beta},#1}}
\newcommand{\V}[1]{\mathbf{W}_{{\beta},#1}}
\newcommand{\la}{\lambda}
\newcommand{\ga}{\gamma}
\newcommand{\ep}{\varepsilon}
\newcommand{\f}{\infty}
\newcommand{\de}{\alpha}

\newcommand{\al}{\alpha}
\newcommand{\si}{\sigma}
\renewcommand{\th}{\theta}
\newcommand{\ra}{\rightarrow}

\begin{document}

\title[]{Hausdorff dimension  of unique beta expansions}

\author{Derong Kong}
\author{Wenxia Li}

\address{Derong Kong: School of Mathematical Science, Yangzhou University,
    Yangzhou, JiangSu 225002, People's Republic of China}

\address{Wenxia Li: Department of Mathematics, East China Normal University, Shanghai 200241,
People's Republic of China}
\email{derongkong@126.com,\quad wxli@math.ecnu.edu.cn}

\begin{abstract}
Given an integer $N\ge 2$ and a real number ${\beta}>1$, let $\Gamma_{{\beta},N}$
be the set of  all $x=\sum_{i=1}^\infty {d_i}/{{\beta}^i}$ with
$d_i\in\{0,1,\cdots,N-1\}$ for all $i\ge 1$. The infinite sequence
$(d_i)$ is called a ${\beta}$-expansion of $x$. Let $\mathbf{U}_{{\beta},N}$ be the set
of all $x$'s in $\Gamma_{{\beta},N}$ which have  unique ${\beta}$-expansions.
We give  explicit formula of the Hausdorff dimension of $\mathbf{U}_{{\beta},N}$ for ${\beta}$  in any admissible interval $[{{\beta}}_L,{{\beta}}_U]$, where ${{\beta}_L}$ is a purely Parry number while ${{\beta}_U}$ is a transcendental number whose quasi-greedy expansion of $1$ is related to the classical Thue-Morse sequence. This allows us to calculate the Hausdorff dimension of $\U{N}$ for almost every $\beta>1$. In particular, this improves  the main results of  G{\'a}bor Kall{\'o}s (1999, 2001).
Moreover, we find that the  dimension function
$f({\beta})=\dim_H\mathbf{U}_{{\beta},N}$  fluctuates frequently for ${\beta}\in(1,N)$.

\medskip


\medskip

\end{abstract}

\maketitle

\section{Introduction}\label{sec: Introduction}

Given an integer $N\ge 2$ and a real number ${\beta}>1$, we call the infinite
sequence $(d_i)$  a \emph{${\beta}$-expansion} of $x$ if we can write
\begin{equation*}
x=\sum_{i=1}^\f\frac{d_i}{{\beta}^i}
\end{equation*}
{with} $d_i\in\{0,1,\cdots,N-1\}$ for all $i\ge 1.$
 Let $\Ga{N}$ be the set of all such $x$'s, i.e.,
\begin{equation*}
\Ga{N}=\Big\{\sum_{i=1}^\f\frac{d_i}{{\beta}^i}: d_i\in\{0,1,\cdots,N-1\}, i\ge 1\Big\}.
\end{equation*}
Then $\Ga{N}$ is a self-similar set generated by
the \emph{iterated function systems }(IFS) $\{f_d(x)=(x+d)/{\beta}:
d\in\{0,1,\cdots,N-1\}\}$ (cf.~\cite{Falconer_1990}).
Let $\{0,1,\cdots,N-1\}^\f$ be the set of all expansions $(d_i)$ with each digit $d_i\in\{0,1,\cdots,N-1\}$.
We define the projection  map $\Pi_{\beta}$ from $\{0,1,\cdots, N-1\}^\f$  to $\Ga{N}$ by
\begin{equation}\label{eq: Pi}
\Pi_{\beta}((d_i))= \sum_{i=1}^\f\frac{d_i}{{\beta}^i}.
\end{equation}
When
${\beta}>N$, the IFS $\{f_d(\cdot): d\in\{0,1,\cdots,N-1\}\}$
satisfies the \emph{strong separation condition} (SSC), and then the map $\Pi_{\beta}$ is bijective which implies that
every point in $\Ga{N}$ has a unique ${\beta}$-expansion. When
${\beta}=N$, the IFS $\{f_d(\cdot): d\in\{0,1,\cdots,N-1\}\}$ fails
the SSC but satisfies the \emph{open set condition} (OSC). Then all
except for countably many points in $\Ga{N}$ have unique
${\beta}$-expansions.

However, when ${\beta}<N$, the IFS $\{f_d(\cdot):
d\in\{0,1,\cdots,N-1\}\}$ fails the OSC. In this case, $\Ga{N}=[0,
(N-1)/({\beta}-1)]$ and almost every point in $\Ga{N}$ have continuum of ${\beta}$-expansions (cf.~\cite{Sidorov_2003,Dajani_DeVries_2007, Sidorov_2007}).
This has close connections to representations of real numbers in non-integer bases. After the seminal works of R\'{e}nyi \cite{Renyi_1957} and Parry \cite{Parry_1960} ${\beta}$-expansions were widely considered from many aspects of mathematics, such as dynamical systems, measure theory, probability, number theory and so on (cf.~\cite{Schmidt_1980,Frougny_Solomyak_1992,Erdos_Joo_Komornik_1990,Dajani_Kraaikamp_2003,Petho_Tichy_1989,Sidorov_2003-1,DeVries_Komornik_2008,Tan_Wang_2011}).

In 1990 Erd\"{o}s, Jo\'{o} and Komornik \cite{Erdos_Joo_Komornik_1990} showed for $N=2$ that for ${\beta}\in(1, G)$ any internal point of $\Ga{N}$ has continuum of ${\beta}$-expansions, and for ${\beta}\in(G,2)$ there exist infinitely many points of $\Ga{N}$ having unique ${\beta}$-expansions (cf.~\cite{Glendinning_Sidorov_2001}), where $G=(1+\sqrt{5})/2$ is the golden ratio.
Recently, Baker \cite{Baker_2012} generalized their result and showed for $N\ge 2$ that there exists $G_N\in(1, N)$ defined by
\begin{equation}\label{eq: gn}
    G_N=\left\{
    \begin{array}{ll}
      k+1& \textrm{if}~ N=2k+1,\\
      \frac{k+\sqrt{k^2+4 k}}{2}& \textrm{if}~ N=2k,
    \end{array}\right.
\end{equation}
such that for each ${\beta}\in(1, G_N)$ any internal point of $\Ga{N}$  has  continuum of ${\beta}$-expansions, and for ${\beta}\in(G_N, N)$ there  exist infinitely many points in $\Ga{N}$ having unique ${\beta}$-expansions (cf.~\cite{Kong_Li_Dekking_2010}).

Let $\U{N}$ be the set of all $x$'s in $\Ga{N}$  which have  unique
${\beta}$-expansions, i.e.,
for any $x\in\U{N}$ there exists a unique sequence $(d_i)\in\{0,1,\cdots,N-1\}^\f$ such that
$
x=\sum_{i=1}^\f d_i/{\beta}^i.
$
When ${\beta}\in(1,N)$, the structure of $\U{N}$ is complex (cf.~\cite{Daroczy_Katai_1993,Darczy_Katai_1995,Glendinning_Sidorov_2001,Kallos_1999,Kallos_2001,Komornik_Loreti_2007,DeVries_Komornik_2008}).
Recently, De Vries and Komornik \cite{DeVries_Komornik_2008} showed
that there exists ${\beta}_c(N)\in(G_N,N)$ such that (see also \cite{Glendinning_Sidorov_2001, Kong_Li_Dekking_2010, DeVries_Komornik_2011})
\begin{itemize}
  \item if ${\beta}\in(G_N,{\beta}_c(N))$, then $|\U{N}|=\aleph_0$;
  \item if ${\beta}={\beta}_c(N)$, then $\dim_H\U{N}=0$ but $|\U{N}|=2^{\aleph_0}$;
  \item if ${\beta}\in({\beta}_c(N),N)$, then $0<\dim_H\U{N}<1$.
\end{itemize}
Here ${\beta}_c(N)$ is the Komornik-Loreti constant defined as the unique positive solution of the equation $
1=\sum_{i=1}^\f \la_i/{\beta}^i, $ where $(\la_i)=(\la_i(N))$ is given
by (cf.~\cite{Komornik_Loreti_2002})
\begin{equation}\label{eq: lambda}
\la_i(N)=\left\{\begin{array}{ll}
  k-1+\tau_i&\textrm{if}~N=2k,\\
  k+\tau_i-\tau_{i-1}&\textrm{if}~N=2k+1,
\end{array}\right.
\end{equation}
with  $(\tau_i)_{i=0}^\f$  the classical Thue-Morse sequence
starting at (cf.~\cite{Allouche_Shallit_1999})
\begin{equation*}
0110\,1001\;1001\,0110\cdots.
\end{equation*}
Allouche and Cosnard \cite{Allouche_Cosnard_2000} showed that ${\beta}_c(2)$ is a transcendental number. Later, Komornik and Loreti \cite{Komornik_Loreti_2002} showed that
${\beta}_c(N)$ is transcendental for any $N\ge 2$.

The purpose of  this paper is to investigate  the Hausdorff dimension of
$\U{N}$. From the above observation it follows that
\begin{itemize}
  \item  if ${\beta}\in(1,{\beta}_c(N)]$, then $\dim_H\U{N}=0$;
  \item if ${\beta}\in[N,\f)$, then $\dim_H\U{N}=\dim_H\Ga{N}=\log N/\log{\beta}$.
\end{itemize}
However, when ${\beta}\in({\beta}_c(N), N)$ we  know little about the Hausdorff dimension of $\U{N}$. When $N=2$, Dar\'{o}czy and K\'{a}tai \cite{Darczy_Katai_1995} gave a method to calculate the Hausdorff dimension of
 $\U{N}$ only if ${\beta}$ is a purely Parry number. When $N>2$, Kall\'{o}s \cite{Kallos_1999} showed that when ${\beta}\in[N-1, (N-1+\sqrt{N^2-2N+5})/2]$ the Hausdorff dimension of $\U{N}$ is given by $
\dim_H\U{N}=\log(N-2)/\log{\beta}.
$
Later in \cite{Kallos_2001} he investigated the Hausdorff dimension of $\U{N}$ for  ${\beta}\in[(N-1+\sqrt{N^2-2N+5})/2, N)$, and gave a method to calculate its Hausdorff dimension when ${\beta}$ is a purely Parry number.

 In this paper we improve the main results of Kall\'{o}s \cite{Kallos_1999, Kallos_2001}. In Theorem \ref{th: main results} we give the Hausdorff dimension of $\U{N}$ for ${\beta}$ in any admissible interval $[{\beta}_L, {\beta}_U]$,  where ${{\beta}_L}$ is a purely Parry number while ${{\beta}_U}$ is a transcendental number.  Moreover, we  show in Theorem \ref{th: main results-1} that all of these admissible intervals cover almost every point of $({\beta}_c(N), N)$. Therefore, we are able to calculate the Hausdorff dimension of $\U{N}$ for almost every $\beta>1$. In particular, we give  explicit formula for the Hausdorff dimension of $\U{N}$ when  ${\beta}$ is in any 1-level or 2-level admissible intervals $[{{\beta}_L},{{\beta}_U}]$ (see Theorem \ref{th: leve1 one} and \ref{th: leve two} for more explanation).

\begin{example}\label{ex£º1}
  Let $N=10$. By Theorem \ref{th: leve1 one}, Theorem \ref{th: leve two} and the above observation we plot in Figure \ref{fig:1} that the graph of the  dimension function
  $f({\beta})=\dim_H\U{10}$  for ${\beta}\in(1,110)$. In particular, we give a detailed plot of $f({\beta})$ for ${\beta}$ in the $1$-level and $2$-level admissible intervals in $({\beta}_c(10), 10)\approx(5.976, 10)$. Clearly, the
dimension function $f({\beta})$ fluctuates frequently for ${\beta}\in({\beta}_c(10), 10)$. In \cite{Komornik_Kong_Li_2014} we will show that $f(\beta)$ is continuous for $\beta>1$.

  \begin{figure}[h!]
{\centering  
  \includegraphics[width=6cm]{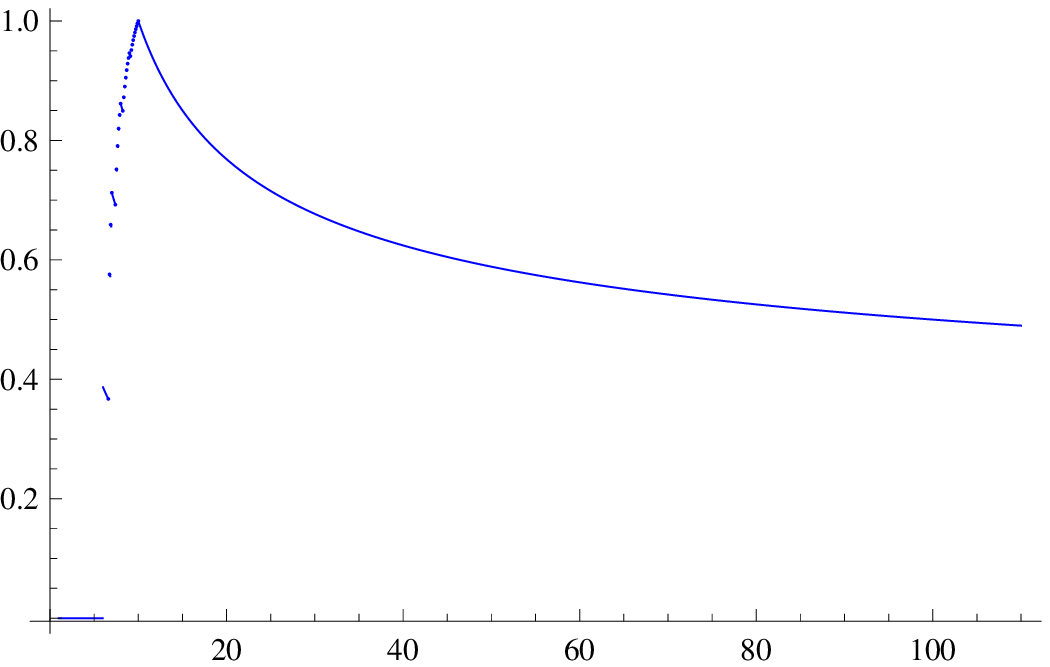}\quad\includegraphics[width=6cm]{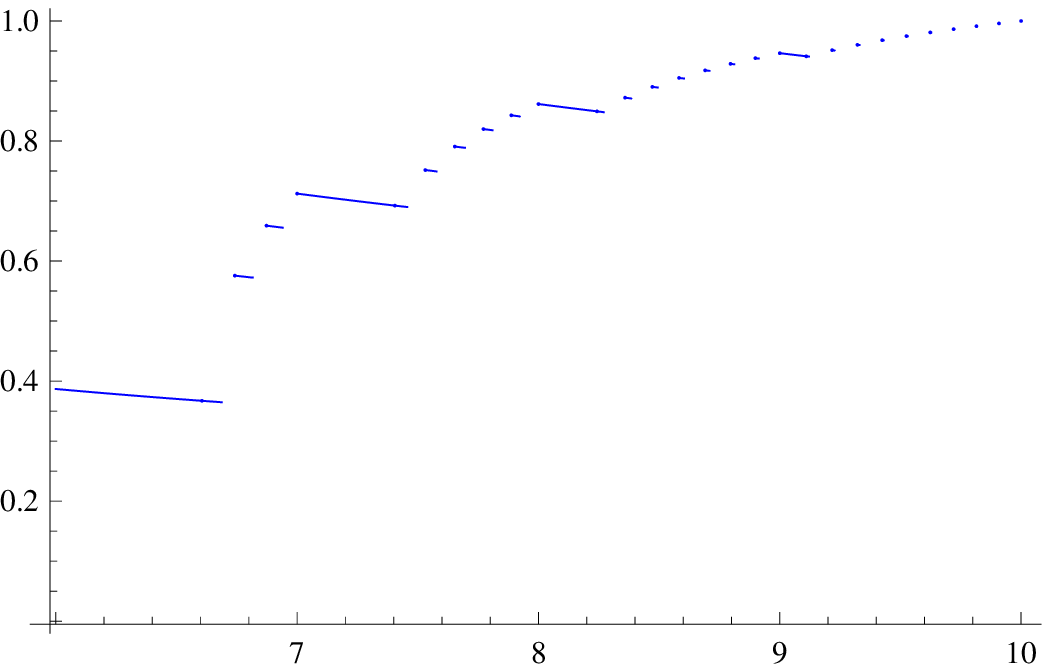}\\
  \caption{The Hausdorff dimension of $\U{10}$ for ${\beta}\in(1, 110)$ (left column), and for ${\beta}$ in the 1-level and 2-level admissible intervals in $({\beta}_c(10),10)\approx(5.976, 10)$ (right column).}\label{fig:1}
}\end{figure}

\end{example}

The structure of the paper is arranged as follows. In Section \ref{sec: main results} we introduce the admissible blocks, admissible intervals and the generalized Thue-Morse sequences, and state our main results as in  Theorems \ref{th: main results-0}, \ref{th: main results-1} and \ref{th: main results}. In Section \ref{sec: unique expansions} we presented some properties of unique beta expansions. The proofs of Theorems \ref{th: main results-0}, \ref{th: main results-1} and \ref{th: main results} are given in Sections \ref{sec: proof of th1}, \ref{sec: proof of th2} and \ref{sec: proof of th3}, respectively. In Section \ref{sec: examples} we consider some examples for which the Hausdorff dimension of $\U{N}$ can be calculated explicitly.

\section{Preliminary and main results}\label{sec: main results}

Given  an integer $N\ge 2$, for any ${\beta}\in(1, N)$   the set $\Ga{N}$  is a closed interval, i.e., $\Ga{N}=[0,(N-1)/({\beta}-1)]$. Then any real number in this interval $\Ga{N}$ has a ${\beta}$-expansion, some  of  them may have multiple ${\beta}$-expansions. Among these expansions
 we define the so-called \emph{greedy ${\beta}$-expansion} $(b_i(x))=(b_i)$ of $x\in\Ga{N}$ recursively as follows (cf.~\cite{Parry_1960}). For $x\in\Ga{N}$, if $b_i$ has already been defined for $1\le i<n$ (no condition if $n=1$), then $b_n$ is the largest element in $\{0,1,\cdots, N-1\}$ satisfying
\begin{equation*}
\frac{b_1}{{\beta}}+\frac{b_2}{{\beta}^2}+\cdots+\frac{b_n}{{\beta}^n}\le x.
\end{equation*}
One can  verify that $(b_i)$ is indeed a ${\beta}$-expansion of
$x$. Moreover, $(b_i)$ is the largest ${\beta}$-expansion of $x$ in the sense of lexicographical order among all ${\beta}$-expansions of $x$.

Accordingly, we define the so-called \emph{quasi-greedy ${\beta}$-expansion} $(a_i(x))=(a_i)$ of $x\in\Ga{N}$ recursively as follows (cf.~\cite{DeVries_Komornik_2008}). For $x=0$ we set $(a_i)=0^\f.$ For $x\in\Ga{N}\setminus\{0\}$, if $a_i$ has already been defined for $1\le i<n$ (no condition if $n=1$), then $a_n$ is the largest element in $\{0,1,\cdots, N-1\}$ satisfying
\begin{equation*}
\frac{a_1}{{\beta}}+\frac{a_2}{{\beta}^2}+\cdots+\frac{a_n}{{\beta}^n}<x.
\end{equation*}
One can  also verify that $(a_i)$ is indeed a ${\beta}$-expansion of
$x$. Clearly, the quasi-greedy ${\beta}$-expansion of $x$ is the
largest infinite ${\beta}$-expansion of $x$ in the sense of lexicographical
order among all ${\beta}$-expansions of $x$. Here  we call a
${\beta}$-expansion \emph{infinite} if the expansion has infinitely many
non-zero elements.

For a positive integer $p$ let $\{0,1,\cdots,N-1\}^p$ be the set of blocks $c_1\cdots c_p$ of length $p$ with each element $c_i\in\{0,1,\cdots,N-1\}$.  For two blocks $c_1\cdots c_p$ and $d_1\cdots d_q$ let $c_1\cdots c_p d_1\cdots d_q\in\{0,1,\cdots,N-1\}^{p+ q}$ denote their concatenation.
In particular,  let $(c_1\cdots c_p)^k$ denote the $k$ times concatenations of  $c_1\cdots c_p$ to itself, and let $(c_1\cdots c_p)^\f$ denote the infinite concatenations of $c_1\cdots c_p$ to itself.
For a digit $c\in\{0,1,\cdots,N-1\}$ its
\emph{reflection} $\overline{c}$ is defined by
\begin{equation*}
\overline{c}:=N-1-c.
\end{equation*}
Accordingly, if $c_i\in\{0,1,\cdots,N-1\}$ for
$i\ge 1$, we shall also write $\overline{c_1\cdots c_n}$ instead of
$\overline{c_1}\cdots\overline{c_n}$, and $\overline{c_1 c_2\cdots}$
instead of $\overline{c_1}\,\overline{c_2}\cdots$.
Finally, for a block $c_1\cdots c_p\in\{0,1,\cdots,N-1\}^{p}$ with $c_p>0$
 we set
\begin{equation*}
 c_1\cdots c_p^-:=c_1\cdots c_{p-1}(c_p-1).
\end{equation*}
Similarly, for a block $c_1\cdots c_p\in\{0,1,\cdots,N-1\}^{p}$ with $c_p<N-1$ we set
\[
c_1\cdots c_p^+:=c_1\cdots c_{p-1}(c_p+1).
\]
In particular, when $p=1$ we set $c_1\cdots c_p^-=c_1^-=c_1-1$ and $c_1\cdots c_p^+=c_1^+=c_1+1$.

 In the sequel we will use
lexicographical order between blocks and sequences.

\begin{definition}\label{def:admissible block}
A block $t_1\cdots t_p\in\{0,1,\cdots,N-1\}^p$ is called an \emph{admissible block} if $t_p<N-1$ and for any $1\le i\le p$ we have
\begin{equation*}
\overline{t_1\cdots t_p}\le t_i\cdots t_pt_1\cdots t_{i-1}\quad\textrm{and}\quad t_i\cdots t_p^+\,\overline{t_1\cdots t_{i-1}}\le{t_1\cdots t_p^+}.
\end{equation*}
\end{definition}
Clearly, there exist infinitely many admissible blocks.
In the following we introduce a generalized     Thue-Morse sequence which plays an essential role in this paper.
\begin{definition}\label{def:generalized Thue-Morse sequence}
For a block $t_1\cdots t_p\in\{0,1,\cdots,N-1\}^p$ with $t_p<N-1$, we  call the sequence $(\theta_i)=(\theta_i(t_1\cdots t_p^+))$ a \emph{generalized Thue-Morse sequence}  generated by the block $t_1\cdots t_p^+$ if $(\theta_i)$ can be defined by induction as follows. First, we set
\begin{equation*}
\theta_1\cdots \theta_p=t_1\cdots t_p^+.
\end{equation*}
Then, if $\theta_1\cdots\theta_{2^m p}$ is already defined for some nonnegative integer $m$, we set
\[
\theta_{2^m p+1}\cdots\theta_{2^{m+1}p}=\overline{\theta_1\cdots\theta_{2^m p}}\,^+.
\]
\end{definition}
We first discovered the generalized Thue-Morse sequences from the work of De Vries and  Komornik \cite{DeVries_Komornik_2008}. Later, we found that these sequences were previously  studied by Allouche and Cosnard \cite{Allouche_Cosnard_1983}, Komornik and Loreti \cite{Komornik_Loreti_2007}, et al.

 If  $N=2k+1$, then the  sequence $(\lambda_i(N))$ defined in Equation (\ref{eq: lambda})  is  exactly the generalized Thue-Morse sequence $(\theta_i(k+1))$. If $N=2k$,  by using Lemma \ref{lem:DK-1} one can also show that
$(\la_i(N))=(\th_i(k)).$ Thus, for any $N\ge 2$ we have
\begin{equation*}
(\la_i(N))=(\th_i(\lceil\frac{N}{2}\rceil)),
\end{equation*}
where $\lceil x\rceil$ denotes the least integer larger than or equal to $x$.

In the remainder of the paper we will  reserve the notation
$(\de_i({\beta}))$ especially for the quasi-greedy ${\beta}$-expansion of
$1\in\Ga{N}=[0,(N-1)/({\beta}- 1)]$ (since ${\beta}\in(1, N)$).

\begin{theorem}\label{th: main results-0}
For $N\ge 2$,  let $t_1\cdots t_p\in\{0,1,\cdots,N-1\}^p$. Then $t_1\cdots t_p$ is an admissible block if and only if $(\alpha_i({{\beta}_L}))=(t_1\cdots t_p)^\f$ and $(\alpha_i({{\beta}_U}))=(\theta_i(t_1\cdots t_p^+))$ for some bases ${{\beta}_L}, {{\beta}_U}\in[G_N, N)$, where $G_N$ is the critical base defined in {\rm(}\ref{eq: gn}{\rm)}.
Moreover, ${{\beta}_L}<{{\beta}_U}$, and ${{\beta}_L}$ is  algebraic while ${{\beta}_U}$ is  transcendental.
\end{theorem}
We point out that Theorem \ref{th: main results-0} generalizes some results in \cite{Allouche_Cosnard_2000} and \cite{Komornik_Loreti_2002}.
Here we call the transcendental numbers ${\beta}_U$   \emph{De Vries-Komornik constants} since these numbers were first studied by De Vries and Komornik in \cite{DeVries_Komornik_2008}.
Later in Proposition \ref{prop:DK-1} and Theorem \ref{th: DK-admissible} we will show that $t_1\cdots t_p$ is an admissible block if and only if $(\alpha_i({\beta}_U))=(\th_i(t_1\cdots t_p^+))$, if and only if $(\th_i(t_1\cdots t_p^+))$ is the unique ${\beta}_U$-expansion of $1$.
\begin{definition}\label{def:admissible interval}
  The closed interval $[{{\beta}_L},{{\beta}_U}]$ given in Theorem \ref{th: main results-0} is called an \emph{admissible interval} generated by $t_1\cdots t_p$ {\rm(}simply called, \emph{admissible interval}{\rm)} if
\begin{equation*}
  (\alpha_i({{\beta}_L}))=(t_1\cdots t_p)^\f\quad\textrm{and}\quad (\alpha_i({{\beta}_U}))=(\th_i(t_1\cdots t_p^+)).
  \end{equation*}
\end{definition}
Since we have infinitely many  admissible intervals, it is worthwhile to investigate the size of union of these admissible intervals and  the relationship between them as well.
\begin{theorem}\label{th: main results-1}
   The union of all admissible intervals covers almost every point of $(\beta_c(N), N)$, where $\beta_c(N)$ is the Komornik-Loreti constant. Moreover, for any two admissible intervals $[{\al_L},{\al_U}]$ and  $[{{\beta}_L},{{\beta}_U}]$, either  $[\al_L,\al_U]\cap[{\beta}_L,{\beta}_U]=\emptyset$ or $\al_U={\beta}_U$.
  \end{theorem}
  Theorem \ref{th: main results-1} says that for any two admissible intervals,  either  they are separated from each other or they have the same right endpoint. Now we state our main result on the Hausdorff dimension of $\U{N}$.
\begin{theorem}\label{th: main results}
For $N\ge 2$, let $[{{\beta}_L}, {{\beta}_U}]$ be an admissible interval generated by $t_1\cdots t_p$.
Then for any ${\beta}\in[{{\beta}_L}, {{\beta}_U}]$ the Hausdorff dimension of $\U{N}$ is given by
\begin{equation*}
 \dim_H\U{N}=\frac{h(Z_{t_1\cdots t_p})}{\log{\beta}},
\end{equation*}
 where $h(Z_{t_1\cdots t_p})$ is the topological entropy of the subshift of finite type
\begin{equation*}
Z_{t_1\cdots t_p}:=\big\{(d_i): \overline{t_1\cdots t_p}\le d_n\cdots d_{n+p-1}\le t_1\cdots t_p, n\ge 1\big\}.
\end{equation*}
\end{theorem}
 We  point out  that when $N=2$ Barrera \cite{Barrera_2013} investigated the topological entropy of $\U{N}$. We also point out  that Theorem \ref{th: main results}  generalizes some results in \cite{Darczy_Katai_1995, Kallos_1999, Kallos_2001, Baatz_Komornik_2011}. This will be explained in Section \ref{sec: examples} via some examples for which the Hausdorff dimension of $\U{N}$ can be calculated explicitly.

\section{Properties of unique expansions}\label{sec: unique expansions}
 Recall that $(\alpha_i({\beta}))$ is the quasi-greedy
 ${\beta}$-expansion of $1$. The following characterization of $(\alpha_i({\beta}))$ can be proved by a slight modification of the proof of \cite[Proposition 2.3]{DeVries_Komornik_2008} (see also, \cite[Theorem 2.2]{Baiocchi_Komornik_2007}).
\begin{proposition}\label{prop:1}
Let $N\ge 2$ and $(\alpha_i({\beta}))$ be the quasi-greedy ${\beta}$-expansion of $1$ w.r.t. the digit set $\{0,1,\cdots,N-1\}$. Then the map ${\beta}\ra(\alpha_i({\beta}))$ is a strictly increasing bijection from the  interval $(1, N]$ onto the set of all infinite sequences $(\ga_i)\in\{0,1,\cdots,N-1\}^\f$ satisfying
\begin{equation*}
  \ga_{k+1}\ga_{k+2}\cdots\le\ga_1\ga_2\cdots\quad\textrm{for all}~k\ge 0.
\end{equation*}
  Moreover, the map ${\beta}\ra(\alpha_i({\beta}))$ is continuous w.r.t. the topology in $\{0,1,\cdots,N-1\}^\f$ induced by the
  metric defined by $d\big((c_i), (d_i)\big)=2^{-\min\{j: c_j\ne
  d_j\}}$.
\end{proposition}

In the sequel we will write $(\alpha_i)$ instead of $(\alpha_i({\beta}))$ for the quasi-greedy ${\beta}$-expansion of $1$ if no confusion arises for ${\beta}$.
  The following proposition for the characterization of greedy expansions can be proved in a similar way as in \cite{Erdos_Joo_Komornik_1990} (see also, \cite[Theorem 3.2]{Baiocchi_Komornik_2007}).
\begin{proposition}\label{prop:4-1}
For $N\ge 2$ and ${\beta}\in(1, N]$, $(b_i(x))=(b_i)$ is the greedy ${\beta}$-expansion of some $x\in[0,(N-1)/({\beta}-1)]$ if and only if
\begin{equation*}
  b_{n+1}b_{n+2}\cdots<\alpha_1\alpha_2\cdots
\end{equation*}
  whenever $b_n<N-1$.
\end{proposition}
By  Proposition \ref{prop:4-1} we have an equivalent  characterization for the greedy expansions (see also~\cite{DeVries_Komornik_2008,Baiocchi_Komornik_2007}).

\begin{proposition}\label{prop:4-2}
    For $N\ge 2$ and ${\beta}\in(1, N]$,  $(b_i)=(b_i(x))$ is the greedy ${\beta}$-expansion of some $x\in[0,(N-1)/({\beta}-1)]$ if and only if
 \begin{equation}\label{eq: RUnique-1}
  b_{n+k+1}b_{n+k+2}\cdots<\alpha_1\alpha_2\cdots
  \end{equation}
  for all $k\ge 0$ whenever $b_n<N-1$.
\end{proposition}
\begin{proof}
  The sufficiency  follows directly by taking $k=0$ in Equation (\ref{eq: RUnique-1}) and then using Proposition \ref{prop:4-1}. For the necessity, suppose $(b_i)$ is the greedy expansion of some $x$, and suppose $b_n<N-1$ for some $n\ge 1$. By Proposition \ref{prop:4-1} we have
\begin{equation}\label{eq: prop4-2}
  b_{n+1}b_{n+2}\cdots<\alpha_1\alpha_2\cdots.
\end{equation}
We claim that
 $b_{n+2}b_{n+3}\cdots<\alpha_1\alpha_2\cdots.$

  If $b_{n+1}<N-1$, Proposition \ref{prop:4-1} yields the claim.
If $b_{n+1}=N-1$, Equation (\ref{eq: prop4-2}) implies that $\alpha_1=N-1$ and therefore
\begin{equation*}
b_{n+2}b_{n+3}\cdots<\alpha_2\alpha_3\cdots\le\alpha_1\alpha_2\cdots,
\end{equation*}
where the second inequality follows from Proposition \ref{prop:1}.

By induction, we have
$b_{n+k+1}b_{n+k+2}\cdots<\alpha_1\alpha_2\cdots$ for all $k\ge 0$.
\end{proof}
Note that an expansion $(d_i)=(d_i(x))$ is the unique expansion of  $x\in\U{N}$ if and only if both $(d_i)$ and $(\overline{d_i})$ are the greedy expansions (cf.~\cite{Erdos_Joo_Komornik_1990}). By using Proposition \ref{prop:4-2} we have the following characterization  of  $\U{N}$.
\begin{theorem}\label{th: 2'}
  For ${\beta}\in(1, N]$, let $(\alpha_i)=(\de_i({\beta}))$ be the quasi-greedy ${\beta}$-expansion of $1$. Then $x\in\U{N}$ if and only if the ${\beta}$-expansion $(d_i)=(d_i(x))$ of $x$ satisfies
  \begin{equation*}
  \left\{
  \begin{array}{l}
    d_{m+k+1}d_{m+k+2}\cdots<\alpha_1\alpha_2\cdots,\\
      \overline{d_{n+k+1}d_{n+k+2}\cdots}<\alpha_1\alpha_2\cdots,
  \end{array}
  \right.
  \end{equation*}
  for all $k\ge 0$, where $m$ is the least integer such that $d_m<N-1$ and $n$ is the least integer such that $d_n>0$.
\end{theorem}
In terms of  Theorem \ref{th: 2'} we can simplify  the calculation of  the Hausdorff dimension of  $\U{N}$ as described in the following theorem.

\begin{theorem}\label{th: 3}
  For $N\ge 2$ and ${\beta}\in(1, N]$, let $(\alpha_i)=(\alpha_i({\beta}))$. Then we have
 \[
  \dim_H\U{N}=\dim_H\V{N},
\]
  where
\begin{equation*}
  \V{N}:=\Big\{\sum_{i=1}^\f\frac{d_i}{{\beta}^i}: ~\overline{\alpha_1\alpha_2\cdots}<d_n d_{n+1}\cdots<\alpha_1\alpha_2\cdots, ~n\ge 1\Big\}.
\end{equation*}
\end{theorem}
\begin{proof}
Clearly, by Theorem \ref{th: 2'} we have $\V{N}\subseteq\U{N}$.
In terms of the properties of Hausdorff dimension it suffices to show that
\begin{equation}\label{eq: RUnique-2}
\begin{split}
  \U{N}&\subseteq\bigcup_{d=1}^{N-2}\frac{d+\V{N}}{{\beta}}\cup\bigcup_{n=1}^\f\bigcup_{d=1}^{N-1}\frac{d+\V{N}}{{\beta}^{n+1}}\\
  &\qquad\qquad\cup\bigcup_{m=1}^\f\bigcup_{d=0}^{N-2}\Big(\sum_{\ell=1}^{m}\frac{N-1}{{\beta}^\ell}+\frac{d+\V{N}}{{\beta}^{m+1}}\Big).
  \end{split}
\end{equation}
Let $x\in\U{N}$ and  $(d_i)=(d_i(x))$ be its unique ${\beta}$-expansion. We will finish the proof by showing in the following three cases that $x$ is also in the right-hand side of  (\ref{eq: RUnique-2}).

Case I. $0<d_1<N-1$. Then by Theorem \ref{th: 2'} it follows that
\[
\overline{\alpha_1\alpha_2\cdots}<d_{k+1}d_{k+2}\cdots<\de_1\de_2\cdots
\]
for all $k\ge 1$, i.e., $d_2 d_3\cdots\in\Pi_{\beta}^{-1}(\V{N})$ where $\Pi_{\beta}$ is the projection map defined in (\ref{eq: Pi}).

Case II. $d_1=0$. Then by Theorem \ref{th: 2'} it yields that
\[
d_{k+1}d_{k+2}\cdots<\de_1\de_2\cdots
\]
for all $k\ge 1$. Let $n$ be the least integer such that $d_n>0$. Again by Theorem \ref{th: 2'} it follows that $d_{n+1}d_{n+2}\cdots\in\Pi_{\beta}^{-1}(\V{N})$.

Case III. $d_1=N-1$. Then in a similar way as in Case II we have $d_{m+1}d_{m+2}\cdots\in\Pi_{\beta}^{-1}(\V{N})$, where $m$ is the least integer such that  $d_m<N-1$.
\end{proof}
Clearly, $\Pi_{\beta}^{-1}(\V{N})$ is a symmetric subshift of $\{0,1,\cdots,N-1\}^\f$. According to Theorem \ref{th: 3} it suffices to prove  Theorem \ref{th: main results} for $\V{N}$ instead of $\U{N}$.

\section{Proof of Theorem \ref{th: main results-0}}\label{sec: proof of th1}
Suppose that $t_1\cdots t_p\in\{0,1,\cdots, N-1\}^p$ is an admissible block. Then by Definition \ref{def:admissible block} it follows
\begin{equation}\label{eq: admissible block}
\overline{t_1\cdots t_p}\le t_i\cdots t_pt_1\cdots t_{i-1}<t_i\cdots t_p^+\,\overline{t_1\cdots t_{i-1}}\le{t_1\cdots t_p^+}
\end{equation}
for any $1\le i\le p$. The following proposition guarantees that $(t_1\cdots t_p)^\f$ is a quasi-greedy expansion of $1$ for some base ${{\beta}_L}\in(1,N]$.
\begin{proposition}\label{prop:DK-0}
  Let $t_1\cdots t_p\in\{0,1,\cdots,N-1\}^p$ be an admissible block. Then $(\alpha_i({{\beta}_L}))=(t_1\cdots t_p)^\f$  for some base ${{\beta}_L}\in(1,N]$.
\end{proposition}
\begin{proof}
 Since $t_1\cdots t_p\in\{0,1,\cdots,N-1\}^p$ is an admissible block, by (\ref{eq: admissible block}) it follows that
\[
  t_i\cdots t_pt_1\cdots t_{i-1}\le t_1\cdots t_p
\]
  for any $1\le i\le p$. This yields
\[
  t_i\cdots t_p(t_1\cdots t_p)^\f=(t_i\cdots t_p t_1\cdots t_{i-1})^\f\le(t_1\cdots t_p)^\f.
\]
  Then by Proposition \ref{prop:1} we have $(\alpha_i({{\beta}_L}))=(t_1\cdots t_p)^\f$ for some ${{\beta}_L}\in(1,N]$.
\end{proof}
Note that $\beta>1$ is   a purely Parry number if $(\al_i(\beta))$ is periodic.
So, the base ${{\beta}_L}$ defined in Proposition \ref{prop:DK-0}  is  a purely Parry number. Later in Proposition \ref{prop:DK-2} we will show that ${\beta}_L\ge G_N$.
Recall from Definition \ref{def:generalized Thue-Morse sequence} that $(\theta_i)=(\theta_i(t_1\cdots t_p^+))$ is a generalized Thue-Morse sequence. We will show in Proposition \ref{prop:DK-1} that if $t_1\cdots t_p$ is admissible, then $(\th_i)$ is also a quasi-greedy expansion of $1$ for some base ${{\beta}_U}$. First we give the following lemma.

\begin{lemma}\label{lem: keylem}
 Let $t_1\cdots t_p$ be an admissible block and  let $(\theta_i)=(\theta_i(t_1\cdots t_p^+))$ be the generalized Thue-Morse sequence generated by $t_1\cdots t_p^+$. Then
 for any $n\ge 0$ we have
\begin{equation}\label{eq: proof-00}
 \overline{\theta_1\cdots\th_{2^n p-i+1}}<   \th_i\cdots\th_{2^np}\le \theta_1\cdots\th_{2^n p-i+1}
\end{equation}
 for any $1\le i\le 2^n p$.
 \end{lemma}
 \begin{proof}
   We will prove (\ref{eq: proof-00}) by using induction on $n$.
Since $t_1\cdots t_p$ is  an admissible block, it follows from Equation (\ref{eq: admissible block}) that
\[
\overline{\th_1\cdots \th_{p-i+1}}\le t_i\cdots t_p<t_i\cdots t_p^+=\theta_i\cdots\th_p\le \th_1\cdots\th_{p-i+1}
\]
for any $1\le i\le p$.
Then (\ref{eq: proof-00}) holds for $n=0$.

Suppose (\ref{eq: proof-00}) holds for $n=k$. We will split the proof of (\ref{eq: proof-00}) for $n=k+1$ into the following two cases.

Case I. $1\le i\le 2^k p$. Then by induction we have $\th_i\cdots\th_{2^kp}>\overline{\th_1\cdots\th_{2^kp-i+1}}$, which yields
\[
\th_i\cdots\th_{2^{k+1} p}>\overline{\th_1\cdots\th_{2^{k+1} p-i+1}}.
\]
Again by induction we have $\th_i\cdots\th_{2^{k} p}\le\th_1\cdots\th_{2^k p-i+1}$, and for any $2\le i\le 2^kp$,
\[
\th_{2^k p+1}\cdots\th_{2^{k} p+i-1}=\overline{\th_1\cdots\th_{i-1}}<\th_{2^k p-i+2}\cdots\th_{2^{k} p},
\]
where  the inequality holds by the induction.  Then
\[
\overline{\th_1\cdots\th_{2^{k+1} p-i+1}}<\th_i\cdots\th_{2^{k+1}p}\le \th_1\cdots\th_{2^{k+1}p-i+1}
\]
for any $1\le i\le 2^k p$.

Case II. $2^k p<i\le 2^{k+1}p$.
Then we can write $i=2^k p+j$ with $1\le j\le 2^k p$. By induction and Definition \ref{def:generalized Thue-Morse sequence} of the generalized Thue-Morse sequence $(\th_i)$ it follows that
$$
\overline{\theta_1\cdots\th_{2^{k+1}p-i+1}}<\overline{\th_j\cdots\th_{2^k p}}\,^+=\th_{i}\cdots\th_{2^{k+1}p}\le\th_1\cdots\th_{2^{k+1}p-i+1}
$$ for any $2^k p<i=2^kp+j\le 2^{k+1}p$.
 \end{proof}

\begin{proposition}\label{prop:DK-1}
 The block $t_1\cdots t_p\in\{0,1,\cdots,N-1\}^p$ is  admissible if and only if  $(\alpha_i({\beta}_U))=(\theta_i(t_1\cdots t_p^+))$  for some base ${{\beta}_U}\in(1,N]$.
\end{proposition}
\begin{proof}
We first prove the sufficiency. Suppose $(\alpha_i({{\beta}_U}))=(\th_i(t_1\cdots t_p^+))$ for some ${\beta}_U\in(1,N]$. By Definition \ref{def:generalized Thue-Morse sequence} the generalized Thue-Morse sequence $(\th_i(t_1\cdots t_p^+))$ begins with
\begin{equation}\label{eq: sec4-3}
(\th_i(t_1\cdots t_p^+))=t_1\cdots t_p^+\,\overline{t_1\cdots t_p}\,\overline{t_1\cdots t_p^+}\,t_1\cdots t_p^+\cdots.
\end{equation}
Then by Proposition \ref{prop:1} it follows that
\[
t_i\cdots t_p^+\overline{t_1\cdots t_{i-1}}\le t_1\cdots t_p^+\quad\textrm{and}\quad \overline{t_i\cdots t_p t_1\cdots t_{i-1}}\le t_1\cdots t_p^+
\]
for any $1\le i\le p$. By Definition \ref{def:admissible block} it suffices to show that $\overline{t_i\cdots t_p t_1\cdots t_{i-1}}\ne t_1\cdots t_p^+$ for all $1\le i\le p$.

Suppose $\overline{t_i\cdots t_p t_1\cdots t_{i-1}}=t_1\cdots t_p^+$ for some $1\le i\le p$. Then by Proposition \ref{prop:1} and Equation (\ref{eq: sec4-3}) it follows that
\[
\overline{t_i\cdots t_p^+}\le\overline{t_1\cdots t_{p-i+1}}.
\]
Observing by Proposition \ref{prop:1} that $t_i\cdots t_p^+\le t_1\cdots t_{p-i+1}$ we obtain
\[\overline{t_i\cdots t_p t_1\cdots t_p^+}=t_1\cdots t_p^+\overline{t_1\cdots t_{p-i+1}}.
\]
This implies that $i\ne 1$. Again  by Proposition \ref{prop:1} and Equation (\ref{eq: sec4-3}) we obtain
\[
t_1\cdots t_{i-1}=\overline{t_{p-i+2}\cdots t_p}\quad\textrm{for}~2\le i\le p.
\]
This leads to a contradiction with the assumption that $\overline{t_i\cdots t_p t_1\cdots t_{i-1}}=t_1\cdots t_p^+$.

In the following we will show the necessity.
Let $i\ge 1$. Then $i<2^n p$ for some large integer $n\ge 0$. By Lemma \ref{lem: keylem} it follows that
$$
\th_{i+1}\cdots\th_{2^n p}\le\th_1\cdots\th_{2^n p-i}\quad\textrm{and}\quad\overline{\th_1\cdots\th_{i}}<\th_{2^n p-i+1}\cdots\th_{2^n p}.
$$
 This implies
\begin{eqnarray*}
\th_{i+1}\cdots\th_{2^n p}\th_{2^n p+1}\cdots\th_{2^np+i}\cdots&=&\th_{i+1}\cdots\th_{2^np}\overline{\th_1\cdots\th_{i}}\cdots\\
&<&\th_1\cdots\th_{2^n p-i}\th_{2^{n}p-i+1}\cdots\th_{2^n p}\cdots.
\end{eqnarray*}
By Proposition \ref{prop:1} this establishes the proposition.
\end{proof}

Moreover, by using Lemma \ref{lem: keylem} one can show that  $(\th_i)$ is   the unique ${\beta}_U$-expansion of $1$.
\begin{theorem}\label{th: DK-admissible}
  Let $t_1\cdots t_p\in\{0,1,\cdots,N-1\}^p$. Then $t_1\cdots t_p$ is an admissible block if and only if the generalized Thue-Morse sequence $(\th_i)=(\th_i(t_1\cdots t_p^+))$ is the unique expansion of $1$ for some base ${{\beta}_U}$, i.e.,
\[
  \overline{\th_1\th_2\cdots}<\th_{i+1}\th_{i+2}\cdots<\th_1\th_2\cdots\quad\textrm{for any}~i\ge 1.
  \]
\end{theorem}
Recall from (\ref{eq: gn}) that $G_N$ is the generalized golden ratio. We will show that the admissible intervals are all included in $[G_N,N)$. In Proposition \ref{prop: maximal admissible interval-cover} we will show that all of these admissible intervals cover $(\beta_c(N), N)$ a.e., where $\beta_c(N)(>G_N)$ is the Komornik-Loreti constant.
\begin{proposition}\label{prop:DK-2}
  Let $[{{\beta}_L},{{\beta}_U}]$ be an admissible interval generated by  $t_1\cdots t_p$. Then $[{{\beta}_L},{{\beta}_U}]\subseteq[G_N, N)$.
\end{proposition}
\begin{proof}
  Clearly, by Definition \ref{def:generalized Thue-Morse sequence} of the generalized Thue-Morse sequence $(\th_i)=(\th_i(t_1\cdots t_p^+))$ it follows that
\[
  (\alpha_i({\beta}_U))=(\th_i)<(N-1)^\f=(\alpha_i(N)).
  \]
By Proposition \ref{prop:1} this implies  ${{\beta}_U}<N$. In the following we will show ${{\beta}_L}\ge G_N$.

 Since $t_1\cdots t_p$ is admissible, it yields that $t_1\ge\overline{t_1}=N-1-t_1$. Then $t_1\ge \lceil(N-1)/2\rceil$.
 By Definition \ref{def:admissible block} of an admissible block one can directly verify that $(t_1\cdots t_p)^\infty \geq (t_1\overline{t_1})^\infty$ (see also \cite[Proposition 2]{Allouche_Frougny_2009}).
  Note by (\ref{eq: gn}) that
$
(\alpha_i(G_N))=(t_1\cdots t_p)^\f=(\lceil(N-1)/2\rceil\overline{\lceil(N-1)/2\rceil})^\f.
$
Then
\[
  (\alpha_i({{\beta}_L}))=(t_1\cdots t_p)^\f\ge (t_1\overline{t_1})^\f\ge(\alpha_i(G_N)).
  \]
  By Proposition \ref{prop:1} this implies ${{\beta}_L}\ge G_N$.
\end{proof}
In the following we will investigate the algebraic properties of the generalized Thue-Morse sequences $(\th_i)$ and show that the De Vries-Komornik constant ${{\beta}_U}$ is transcendental.
 Recall that $(\tau_i)_{i=0}^\f$ is the classical Thue-Morse sequence beginning with
\[
0110\,1001\,1001\,0110\,1001\,0110\,0110\,1001\cdots.
\]
We write two equivalent defintions for this sequence $(\tau_i)$ (see, e.g., \cite{Allouche_Shallit_1999} for details).
\renewcommand{\labelenumi}{(\Roman{enumi})}
\begin{enumerate}
  \item Set $\tau_0=0, \tau_{2^{n}}=1$ for $n=0,1,\cdots,$ and
\[
\tau_{2^n+k}=1-\tau_k\quad\textrm{if}\quad 1\le k<2^n, n=1,2,\cdots.
\]
  \item For a nonnegative integer $i$ we consider its dyadic expansion
\[
i=\ep_n 2^n+\ep_{n-1}2^{n-1}+\cdots+ \ep_0,\quad \ep_k\in\{0,1\}.
\]
Then we set
\[
\tau_i=\left\{\begin{array}{ll}
        0 & \textrm{if}~ \sum_{j=0}^n\ep_j~ \textrm{is even}, \\
         1& \textrm{if}~ \sum_{j=0}^n\ep_j ~\textrm{is odd}.
       \end{array}
       \right.
\]
\end{enumerate}
Based on $(\tau_i)$ we give an equivalent definition for the generalized Thue-Morse sequence $(\th_i)$.
\begin{lemma}\label{lem:DK-1}
Let $(\theta_i)=(\theta_i(t_1\cdots t_p^+))$ be the generalized Thue-Morse sequence generated by  $t_1\cdots t_p^+$. Then for any integer  $\ell=i p+q$ with $i\ge 0, 1\le q\le p$ we have
\begin{equation}\label{eq:DK-2}
\theta_\ell=\left\{
\begin{array}{ll}
  t_q+\tau_i(\overline{t_q}-t_q), & \textrm{if} \quad 1\le q<p \\
  t_q+\tau_i(\overline{t_q}-t_q)+(\tau_{i+1}-\tau_i), & \textrm{if} \quad q=p.
\end{array}
\right.
\end{equation}
\end{lemma}
\begin{proof}
Recall from  Definition \ref{def:generalized Thue-Morse sequence} that $(\eta_i)$ is the generalized Thue-Morse sequence generated by $t_1\cdots t_p^+$ if and only if
$\eta_1\cdots\eta_p=t_1\cdots t_p^+$, and for any $n\ge 0$ we have
\begin{equation}\label{eq: condition for DK}
  \eta_{2^{n+1}p}=\overline{\eta_{2^n p}}+1=N-\eta_{2^n  p},\quad  \eta_{2^n p+k}=\overline{\eta_k}\quad\textrm{for all}~ 1\le k<2^n p.
   \end{equation}
Clearly, by using $\tau_0=0, \tau_1=1$ in Equation (\ref{eq:DK-2}) it yields that $\theta_1\cdots\theta_p=t_1\cdots t_p^+$. Then it suffices to show that the sequence $(\theta_i)$ given in Equation (\ref{eq:DK-2}) satisfies the  conditions in (\ref{eq: condition for DK}).

For $n\ge 0$, by using Definition (I) of $(\tau_i)$  and Equation (\ref{eq:DK-2}) it follows that
\begin{eqnarray*}
  &&\theta_{2^{n+1}p}+\theta_{2^n p}\\
  &=&\big(t_p+\tau_{2^{n+1}-1}(\overline{t_p}-t_p)+(\tau_{2^{n+1}}-\tau_{2^{n+1}-1})\big)\\
  &&+\big(t_p+\tau_{2^{n}-1}(\overline{t_p}-t_p)+(\tau_{2^{n}}-\tau_{2^{n}-1})\big)\\
&=&t_p+(1-\tau_{2^{n}-1})(\overline{t_p}-t_p)+(1-(1-\tau_{2^n-1}))\\
&&+t_p+\tau_{2^{n}-1}(\overline{t_p}-t_p)+(1-\tau_{2^{n}-1})\\
&=&t_p+\overline{t_p}+1=N,
  \end{eqnarray*}
i.e., $\theta_{2^{n+1}p}=N-\theta_{2^n  p}=\overline{\theta_{2^n p}}+1$.

For $1\le k<2^n p$ we can write
$
k=(\ep_{n-1}2^{n-1}+\cdots+\ep_1 2^1+\ep_0)p+q$ with $ \ep_{n-1},\cdots,\ep_0\in\{0,1\}$ and $ 1\le q\le p$. Without loss of generality we may assume $1\le q<p$.
If $\sum_{j=0}^{n-1}\ep_j$ is even, then by using Definition (II) of $(\tau_i)$ and Equation (\ref{eq:DK-2}) it follows that
\[
\theta_k=\theta_{(\sum_{j=0}^{n-1}\ep_j 2^j)p+q}=t_q+0(\overline{t_q}-t_q)=t_q,
\]
and
\[
\theta_{2^n p+k}=\theta_{(2^n+\sum_{j=0}^{n-1}\ep_j 2^j)p+q}=t_q+1(\overline{t_q}-t_q)=\overline{t_q}.
\]
So, $\theta_{2^n p+k}=\overline{\theta_k}$. Similarly, if $\sum_{j=0}^{n-1}\ep_j$ is odd , one can also show that $\theta_{2^n p+k}=\overline{\theta_k}$.
\end{proof}
 The following theorem for transcendental numbers is due to Mahler \cite{Mahler_1976} (see also \cite{Komornik_Loreti_2002}).
\begin{theorem}[Mahler \cite{Mahler_1976}]\label{th:Mahler}
  If $z$ is an algebraic number in the open unit disc, then the number
\[
  Z:=\sum_{i=1}^\f \tau_i z^i
  \]
  is transcendental, where $(\tau_i)$ is the classical Thue-Morse sequence.
\end{theorem}

\begin{proof}[Proof of Theorem \ref{th: main results-0}]
Clearly, by Proposition \ref{prop:1} ${{\beta}_L}<{{\beta}_U}$.
By Propositions \ref{prop:DK-0}, \ref{prop:DK-1} and \ref{prop:DK-2} it remains to show that the De Vries-Komornik constant ${{\beta}_U}$ is transcendental.

Let $(\theta_\ell)=(\theta_\ell(t_1\cdots t_p^+))$ be the generalized Thue-Morse sequence generated by the block $t_1\cdots t_p^+$. By the definition of ${{\beta}_U}$ we have
\[
1=\sum_{\ell=1}^\f\theta_\ell{{\beta}_U}^{-\ell}.
\]
For any integer $\ell\ge 1$, let $\ell=i p+q$ with $i\ge 0$ and $1\le q\le p$.  Then by using Lemma \ref{lem:DK-1} we can rewrite the above equation as follows.
\begin{eqnarray*}
  1&=&\sum_{\ell=1}^\f \theta_\ell{{\beta}_U}^{-\ell}=\sum_{i=0}^\f\sum_{q=1}^p \theta_{ip+q}{{\beta}_U}^{-ip-q}\\
  &=&\sum_{i=0}^\f{{\beta}_U}^{-ip}\Big(\sum_{q=1}^p\big(t_q+\tau_i(\overline{t_q}-t_q)\big){{\beta}_U}^{-q}+(\tau_{i+1}-\tau_i){{\beta}_U}^{-p}\Big)\\
  &=&\sum_{i=0}^\f{{\beta}_U}^{-ip}\Big(\sum_{q=1}^p t_q{{\beta}_U}^{-q}\Big)+\sum_{i=0}^\f\tau_i{{\beta}_U}^{-ip}\Big(\sum_{q=1}^p (\overline{t_q}-t_q){{\beta}_U}^{-q}\Big)\\
  &&+\sum_{i=0}^\f(\tau_{i+1}-\tau_i){{\beta}_U}^{-ip-p}\\
  &=&\frac{\sum_{q=1}^p t_q{{\beta}_U}^{-q}}{1-{{\beta}_U}^{-p}} +\Big(\sum_{i=1}^\f \tau_i({{\beta}_U}^{-p})^{i}\Big)\Big(\sum_{q=1}^p (\overline{t_q}-t_q){{\beta}_U}^{-q}\Big)\\
  &&+\sum_{i=1}^\f \tau_i({{\beta}_U}^{-p})^{i}-{{\beta}_U}^{-p}\sum_{i=1}^\f \tau_i({{\beta}_U}^{-p})^{i},
  \end{eqnarray*}
where the last equality holds since $\tau_0=0$. Rearranging  the above equation it gives
\[
\sum_{i=1}^\f \tau_i({{\beta}_U}^{-p})^{i}=\frac{1-{{\beta}_U}^{-p}-\sum_{q=1}^p t_q{{\beta}_U}^{-q}}{(1-{{\beta}_U}^{-p})\Big(1-{{\beta}_U}^{-p}+\sum_{q=1}^p(\overline{t_q}-t_q){{\beta}_U}^{-q}\Big)}.
\]
If ${{\beta}_U}>1$ is an algebraic number, then the right-hand side would be algebraic, while the left hand side would be transcendental by Theorem \ref{th:Mahler}. This contradiction implies that ${{\beta}_U}$ is transcendental.
\end{proof}

\section{Proof of Theorem \ref{th: main results-1}}\label{sec: proof of th2}
First we will show that all of the admissible intervals cover almost every point of $(\beta_c(N), N)$. Let $\mathbf{U}$ be the set of  $\beta\in(1,N]$ for which $1\in\Ga{N}$ has a unique  $\beta$-expansion, i.e., there exists a unique sequence $(d_i)\in\{0,1,\cdots,N-1\}^\f$ such that $1=\sum_{i=1}^\f d_i/\beta^i$. Let $\overline{\mathbf{U}}$ be the closure of $\mathbf{U}$. The following proposition for $\overline{\mathbf{U}}$ was first proved by Komornik and Loreti  \cite{Komornik_Loreti_2007} for  $\beta\in[N-1, N]$ and recently proved by Komornik et al. in \cite{Komornik_Kong_Li_2014}.
\begin{proposition}\label{prop: unique expansion of 1}
 For  $\beta\in(1,N]$ let $(\al_i)=(\al_i(\beta))$ be the quasi-greedy $\beta$-expansion of $1$. Then $\beta\in \overline{\mathbf{U}}$ if and only if
 \[
  \overline{\al_1\al_2\cdots}<\al_{k+1}\al_{k+2}\cdots\le \al_1\al_2\cdots\quad\textrm{for all} ~ k\ge 0.
\]
  Moreover, $\overline{\mathbf{U}}$ has zero Lebesgue measure.
\end{proposition}

\begin{proposition}
  \label{prop: maximal admissible interval-cover}
  The union of all  admissible intervals covers   $(\beta_c(N), N)$ a.e..
\end{proposition}
\begin{proof}
By Proposition \ref{prop: unique expansion of 1} it suffices to show that  $(\beta_c(N),N)$  is  covered by $\overline{\mathbf{U}}$ and   the union of all admissible intervals. Take $\beta\in(\beta_c(N), N)$, and
  let $({\alpha}_i)=({\alpha}_i(\beta))$ be the quasi-greedy $\beta$-expansion of $1$. By Proposition \ref{prop:1} it gives
\begin{equation}\label{eq: sec4-2'}
  {\alpha}_{k+1}{\alpha}_{k+2}\cdots\le {\alpha}_1{\alpha}_2\cdots\quad\textrm{for any}~k\ge 0.
\end{equation}
  Suppose $\beta\notin\overline{\mathbf{U}}$. By Proposition \ref{prop: unique expansion of 1} it follows that there exists $q\ge 0$ such that
  \begin{equation}
    \label{eq: sec4-3}
    {\alpha}_{q+1}{\alpha}_{q+2}\cdots\le \overline{{\alpha}_1{\alpha}_2\cdots}.
  \end{equation}
Let $m$ be the least integer $q$ satisfying (\ref{eq: sec4-3}).  Since $\beta>\beta_c(N)$, by Proposition \ref{prop:1} we have $\al_1>\overline{\al_1}$. Then $m\ge 1$, and one can verify that $\al_m>0$. We will finish the proof by showing that $\beta$ is contained in the admissible interval $[\beta_L, \beta_U]$ generated by ${\alpha}_1\cdots{\alpha}_m^-$.

First we will show the admissibility of ${\alpha}_1\cdots {\alpha}_m^-$. Since $\beta>\beta_c(N)$, by Proposition \ref{prop:1} it follows that either $\al_1^-\ge \overline{\al_1^-}$ or  $\al_2\ge k=\al_1>\overline{\al_1}$ with $N=2k$. If $m=1$, then by the definition of $m$  it gives that $\al_1^-\ge\overline{\al_1^-}$. This yields the admissibility of $\al_1^-$.
In the following we will assume $m\ge 2$.  Since $m$ is the least integer satisfying (\ref{eq: sec4-3}),  it follows that
 \begin{equation*}
   \label{eq: sec4-5}
   {\alpha}_{i}\cdots{\alpha}_m\ge\overline{{\alpha}_1\cdots{\alpha}_{m-i+1}}\quad\textrm{for any}~1\le i\le m.
 \end{equation*}
 We claim that ${\alpha}_{i}\cdots {\alpha}_m>\overline{{\alpha}_1\cdots{\alpha}_{m-i+1}}$ for any $1\le i\le m$.

Suppose ${\alpha}_{i}\cdots{\alpha}_m=\overline{{\alpha}_1\cdots{\alpha}_{m-i+1}}$ for some $1\le i\le m$. Then by the minimality of $m$ and (\ref{eq: sec4-2'}) we have
\[
{\alpha}_{m+1}{\alpha}_{m+2}\cdots>\overline{{\alpha}_{m-i+2}{\alpha}_{m-i+3}\cdots}\ge\overline{{\alpha}_1{\alpha}_2\cdots},
\]
leading to a contradiction with (\ref{eq: sec4-3}).

So, $\alpha_i\cdots\alpha_m>\overline{\alpha_1\cdots\alpha_{m-i+1}}$ for any $1\le i\le m.$
This, together with (\ref{eq: sec4-2'}),
implies that
\[
\overline{\al_1\cdots \al_m^-}\le \al_i\cdots\al_m^-\al_1\cdots \al_{i-1}\quad\textrm{and}\quad\al_i\cdots\al_m \overline{\al_1\cdots\al_{i-1}}\le\al_1\cdots\al_m,
\] for any $1\le i\le m$.
By Definition \ref{def:admissible block} ${\alpha}_1\cdots{\alpha}_m^-$ is admissible.

Now we will show that $\beta\in[\beta_L, \beta_U]$ with $({\alpha}_i(\beta_L))=({\alpha}_1\cdots{\alpha}_m^-)^\f$ and $({\alpha}_i(\beta_U))=(\th_i({\alpha}_1\cdots{\alpha}_m))$. This can be verified by using Proposition \ref{prop:1} in the following equation.
\[
({\alpha}_1\cdots{\alpha}_m^-)^\f<{\alpha}_1{\alpha}_2\cdots<{\alpha}_1\cdots{\alpha}_m\,\overline{{\alpha}_1\cdots{\alpha}_m}\,^+\cdots=(\th_i({\alpha}_1\cdots{\alpha}_m)),
\]
where the second inequality follows by (\ref{eq: sec4-3}).
\end{proof}
  By Theorem \ref{th: main results} and the proof of Proposition \ref{prop: maximal admissible interval-cover} we are able to calculate the Hausdorff dimension of $\U{N}$ for any $\beta\in(\beta_c(N),N)\setminus\overline{\mathbf{U}}$.
\begin{corollary}\label{cor: 1}
 For $\beta\in (\beta_c(N),N)\setminus\overline{\mathbf{U}}$, let $({\alpha}_i)=({\alpha}_i(\beta))$  and let $m$ be the least integer satisfying
\[
  {\alpha}_{m+1}{\alpha}_{m+2}\cdots\le \overline{{\alpha}_1{\alpha}_2\cdots}.
  \]
  Then $\dim_H\U{N}= h(Z_{{\alpha}_1\cdots {\alpha}_m^-})/\log\beta$, where $h(Z_{{\alpha}_1\cdots{\alpha}_m^-})$ is the topological entropy of
\[
  Z_{{\alpha}_1\cdots{\alpha}_m^-}=\{(d_i): \overline{{\alpha}_1\cdots{\alpha}_m^-}\le d_n \cdots d_{n+m-1}\le {\alpha}_1\cdots{\alpha}_m^-, n\ge 1\}.
  \]
\end{corollary}

In the following we will investigate the relationship between any two  admissible intervals. Let $[{\al_L},{\al_U}]$ and $[{{\beta}_L},{{\beta}_U}]$ be two admissible intervals generated by $s_1\cdots s_q$ and $t_1\cdots t_p$ respectively. Then by Definition \ref{def:admissible interval}
\[
(\alpha_i({\al_L}))=(s_1\cdots s_q)^\f,\quad(\alpha_i({\al_U}))=(\th_i(s_1\cdots s_q^+)),
\]
and
\[
(\alpha_i({{\beta}_L}))=(t_1\cdots t_p)^\f,\quad (\alpha_i({{\beta}_U}))=(\th_i(t_1\cdots t_p^+)).
\]
We will prove that $\al_L<{\beta}_U$ implies $\al_U\le{\beta}_U$. By Proposition \ref{prop:1} this is equivalent to showing
\begin{equation}\label{eq: sec5-1}
(s_1\cdots s_q)^\f<(\th_i(t_1\cdots t_p^+))\quad\Longrightarrow\quad
 (\th_i(s_1\cdots s_q^+))\le(\th_i(t_1\cdots t_p^+)).
\end{equation}
We will split the proof of (\ref{eq: sec5-1}) into the following two cases: Case I. $1\le q<p$ (see Lemma \ref{lem: sect6-1}); Case II. $q\ge p$ (see Lemma \ref{prop:6-1}).
 First we give the following lemma.
\begin{lemma}\label{lem: sect6-0}
  Let $t_1\cdots t_p$ be an admissible block. Then for any $q<p/2$ we have $\overline{t_1\cdots t_q}\,^+\le t_{q+1}\cdots t_{2q}$.
\end{lemma}
\begin{proof}
 Suppose
$
\overline{t_1\cdots t_q}\,^+>t_{q+1}\cdots t_{2q}
$
  for some $q<p/2$. Write $p=m2q+j$ with $m\ge 1$ and $0<j\le 2q$. Since $t_1\cdots t_p$ is an admissible block, by (\ref{eq: admissible block}) it yields that $t_{q+1}\cdots t_{2q}\ge\overline{t_1\cdots t_q}$. So,
\[
 t_1\cdots t_{2q}=t_1\cdots t_q\,\overline{t_1\cdots t_q}.
 \]
 Again by (\ref{eq: admissible block}) it follows that
\[
 t_{q+1}\cdots t_{3q}=\overline{t_1\cdots t_q}\,t_{2q+1}\cdots t_{3q}\ge\overline{t_1\cdots t_{2q}}=\overline{t_1\cdots t_q}t_1\cdots t_q,
  \]
  and
  $t_{2q+1}\cdots t_{3q}\le t_1\cdots t_q$. This yields $t_{2q+1}\cdots t_{3q}=t_1\cdots t_q$.

  By iteration, one can show that
\[
  t_1\cdots t_p=t_1\cdots t_{m 2q+j}=(t_1\cdots t_q\,\overline{t_1\cdots t_q})^m t_1\cdots t_j=(t_1\cdots t_{2q})^m t_1\cdots t_j.
  \]
This is impossible since otherwise we have  by (\ref{eq: admissible block})  that
\[
  t_1\cdots t_j^+=t_{p-j+1}\cdots t_{p}^+\le t_1\cdots t_j.
  \]
\end{proof}

\begin{lemma}\label{lem: sect6-1}
Let $s_1\cdots s_{q}$ and $t_1\cdots t_p$ be two admissible blocks with $1\le q< p$. If  $(s_1\cdots s_q)^\f<(\th_i(t_1\cdots t_p^+))$, then
 $
 (\th_i(s_1\cdots s_q^+))\le(\th_i(t_1\cdots t_p^+)).
 $
\end{lemma}
\begin{proof}
Suppose
\begin{equation}\label{eq: sec6-1}
  (s_1\cdots s_q)^\f<(\th_i(t_1\cdots t_p^+))=:(\eta_i).
\end{equation}
Then $s_1\cdots s_q\le \eta_1\cdots\eta_q$. We claim that $s_1\cdots s_q<\eta_1\cdots\eta_q$.

If $s_1\cdots s_q= \eta_1\cdots\eta_q$, then by (\ref{eq: sec6-1}) and Theorem \ref{th: DK-admissible} it follows that
\[
s_1\cdots s_q\le \eta_{q+1}\cdots\eta_{2q}\le\eta_{1}\cdots\eta_{q}=s_1\cdots s_q.
\]
This yields $\eta_1\cdots\eta_{2q}=(s_1\cdots s_q)^2$. By iteration, we have $(\eta_i)=(s_1\cdots s_q)^\f$, leading to a contradiction with (\ref{eq: sec6-1}).

So, $s_1\cdots s_q<\eta_1\cdots\eta_q$, i.e., $s_1\cdots s_q^+\le\eta_1\cdots\eta_q$. Set
$
(\xi_i):=(\th_i(s_1\cdots s_q^+)).
$
Clearly, if  $\xi_1\cdots\xi_q=s_1\cdots s_q^+<\eta_1\cdots \eta_q$, then $(\xi_i)<(\eta_i)$.
  Now we assume
\[
\xi_1\cdots\xi_q=\eta_1\cdots\eta_q\quad\textrm{and}\quad p=2^n q+j
\]
with $n\ge 0$ and $0< j\le 2^nq$. We will split the proof of $(\xi_i)\le (\eta_i)$ into the following two cases.

 Case I. $n=0$. Then $p=q+j$ for $0<j\le q$.  By Definition  \ref{def:generalized Thue-Morse sequence}  and Lemma \ref{lem: keylem} it follows that for $0<j<q$
\[
\xi_{q+1}\cdots \xi_{q+j}=\overline{\xi_1\cdots\xi_j}=\overline{\eta_1\cdots \eta_j}<\eta_{p-j+1}\cdots\eta_p=\eta_{q+1}\cdots\eta_{q+j},
\]
and for $j=q$,
\[
\xi_{q+1}\cdots\xi_{2q}=\overline{\xi_1\cdots\xi_q}\,^+=\overline{\eta_1\cdots\eta_q}\,^+\le\eta_{q+1}\cdots\eta_{2q}.
\]
Then by Definition \ref{def:generalized Thue-Morse sequence} we obtain  that $(\xi_i)\le(\eta_i)$.

Case II. $n\ge 1$. Then $q<p/2$. By Definition \ref{def:generalized Thue-Morse sequence} and Lemma \ref{lem: sect6-0} it follows that
\[
\xi_{1}\cdots\xi_{2q}=\xi_1\cdots\xi_q\overline{\xi_1\cdots\xi_q}\,^+=\eta_1\cdots \eta_q\overline{\eta_1\cdots \eta_q}\,^+\le \eta_{1}\cdots \eta_{2q}.
\]
If  $\xi_1\cdots \xi_{2q}<\eta_1\cdots\eta_{2q}$, then  $(\xi_i)<(\eta_i)$. Suppose $\xi_1\cdots \xi_{2q}=\eta_1\cdots\eta_{2q}$. Then by iteration we have
\[
\xi_1\cdots\xi_{2^nq}\le\eta_1\cdots\eta_{2^nq}.
\]
Clearly, if
$
\xi_1\cdots\xi_{2^nq}<\eta_1\cdots\eta_{2^nq},
$
then  $(\xi_i)<(\eta_i)$. Now suppose
 $\xi_1\cdots\xi_{2^nq}=\eta_1\cdots\eta_{2^nq}$.
In a similar way as in Case I, one can show by Definition \ref{def:generalized Thue-Morse sequence}  and Lemma \ref{lem: keylem}  that
$
\xi_{2^nq+1}\cdots\xi_{2^nq+j}<\eta_{2^nq+1}\cdots\eta_{2^nq+j}
$
if $0<j<2^n q$, and
$
\xi_{2^nq+1}\cdots\xi_{2^{n+1}q}\le\eta_{2^nq+1}\cdots\eta_{2^{n+1}q}
$
if $p=2^{n+1}q.$
Then $(\xi_i)\le(\eta_i)$.
\end{proof}

\begin{lemma}\label{prop:6-1}
Let $s_1\cdots s_q$ and $t_1\cdots t_p$ be two admissible blocks with $q\ge p$. If $(s_1\cdots s_q)^\f<(\th_i(t_1\cdots t_p^+))$, then
 $
 (\th_i(s_1\cdots s_q^+))\le(\th_i(t_1\cdots t_p^+)).
 $
\end{lemma}
\begin{proof}
Suppose
\begin{equation}\label{eq: sec6-2}
  (s_1\cdots s_q)^\f<(\th_i(t_1\cdots t_p^+))=(\eta_i)\quad\textrm{and}\quad q=2^n p+j
\end{equation}
with $n\ge 0$ and $0\le j<2^n p$.
Then
$
s_1\cdots s_{2^n p}\le \eta_1\cdots \eta_{2^np}.
$
If $s_1\cdots s_{2^n p}< \eta_1\cdots \eta_{2^np},$ then by Definition \ref{def:generalized Thue-Morse sequence}  it follows that
\[
(\th_i(s_1\cdots s_q^+))=s_1\cdots s_{2^n p}s_{2^n p+1}\cdots<(\eta_i)=(\th_i(t_1\cdots t_p^+)).
\]
We will finish the proof by showing that $s_1\cdots s_{2^n p}\ne\eta_1\cdots \eta_{2^n p}$.

Suppose $s_1\cdots s_{2^n p}=\eta_1\cdots \eta_{2^n p}$. We claim that
\begin{equation}\label{eq: sec6-5}
s_1\cdots s_q=\eta_1\cdots \eta_q.
\end{equation}
Clearly, if $j=0$, i.e., $q=2^n p$, then (\ref{eq: sec6-5}) holds. Now we assume $0<j<2^n p$.
By (\ref{eq: sec6-2}) and Definition \ref{def:generalized Thue-Morse sequence} it follows that
\[
s_{2^n p+1}\cdots s_{2^n p+j}\le\eta_{2^np+1}\cdots\eta_{2^np+j}=\overline{\eta_1\cdots\eta_j}.
\]
Also by the admissibility of $s_1\cdots s_q$ we have
\[
s_{2^n p+1}\cdots s_{2^n p+j}\ge \overline{s_1\cdots s_j}=\overline{\eta_1\cdots\eta_j}.
\]
Then $s_{2^n p+1}\cdots s_{2^n p+j}=\overline{\eta_1\cdots\eta_j}=\eta_{2^np+1}\cdots\eta_{2^n p+j}$ which yields
Equation (\ref{eq: sec6-5}).

By using Equation (\ref{eq: sec6-5}) in (\ref{eq: sec6-2}) it follows from Theorem \ref{th: DK-admissible} that
\[
s_1\cdots s_q\le \eta_{q+1}\cdots\eta_{2q}\le\eta_1\cdots\eta_q=s_1\cdots s_q.
\]
Then $\eta_{q+1}\cdots\eta_{2q}=s_1\cdots s_q$.
 By iteration, we have
 \[
 (\th_i(t_1\cdots t_p^+))=(\eta_i)=(s_1\cdots s_q)^\f,
 \]
  leading to a contradiction with (\ref{eq: sec6-2}).
\end{proof}
In the following we will prove that  ${\al_L}>{{\beta}_L}$ implies ${\al_U}\ge{{\beta}_U}$. By Proposition \ref{prop:1} this  is equivalent to showing
\begin{equation}\label{eq: sec5-2}
(s_1\cdots s_q)^\f>(t_1\cdots t_p)^\f\quad\Longrightarrow\quad
 (\th_i(s_1\cdots s_q^+))\ge(\th_i(t_1\cdots t_p^+)).
\end{equation}
 The proof of (\ref{eq: sec5-2}) will also be split into the following two cases: Case I. $1\le q<p$ (see Lemma \ref{lem: sect6-2}); Case II. $q\ge p$ (see Lemma \ref{prop:6-2}).
\begin{lemma}\label{lem: sect6-2}
Let $s_1\cdots s_q$ and $t_1\cdots t_p$ be two admissible blocks with $1\le q<p$. If $(s_1\cdots s_q)^\f>(t_1\cdots t_p)^\f$, then
 $
 (\th_i(s_1\cdots s_q^+))\ge(\th_i(t_1\cdots t_p^+)).
 $
\end{lemma}
\begin{proof}
Suppose
  \begin{equation}\label{eq: sec6-3}
    (s_1\cdots s_q)^\f>(t_1\cdots t_p)^\f\quad\textrm{and}\quad p=n q+j
  \end{equation}
  with $n\ge 1$ and $1\le j\le q$.
  Then  $(s_1\cdots s_q)^ns_1\cdots s_j\ge t_1\cdots t_{nq+j}=t_1\cdots t_p$. If $(s_1\cdots s_q)^ns_1\cdots s_j=t_1\cdots t_p$, then
 \[
  t_{p-j+1}\cdots t_{p}=s_1\cdots s_{j}=t_1\cdots t_j,
  \]
  leading to a contradiction with the admissibility of $t_1\cdots t_p$.
  So, $(s_1\cdots s_q)^ns_1\cdots s_j> t_1\cdots t_p$, i.e., $(s_1\cdots s_q)^ns_1\cdots s_j\ge t_1\cdots t_p^+$. Then by Definition \ref{def:generalized Thue-Morse sequence} it follows that
  \[
  (\th_i(s_1\cdots s_q^+))>(s_1\cdots s_q)^n s_1\cdots s_j (N-1)^\f\ge (\th_i(t_1\cdots t_p^+)).
  \]
\end{proof}
When $(s_1\cdots s_q)^\f>(t_1\cdots t_p)^\f$ with $q\ge p$, it is more involved to prove $(\th_i(s_1\cdots s_q^+))\ge(\th_i(t_1\cdots t_p^+)).$ First we consider the following lemma.

\begin{lemma}\label{lem: sec6-3}
 Let $s_1\cdots s_q$ and $t_1\cdots t_p$ be two admissible blocks with $q\ge p$. If $s_1\cdots s_p>t_1\cdots t_p$, then
 $
 (\th_i(s_1\cdots s_q^+))\ge(\th_i(t_1\cdots t_p^+)).
 $
\end{lemma}
\begin{proof}
  Write $q=2^n p+j$ with $n\ge 0$ and $0\le j<2^n p$. Suppose $s_1\cdots s_p>t_1\cdots t_p$, i.e., \[
  s_1\cdots s_p\ge t_1\cdots t_p^+=(\th_i(t_1\cdots t_p^+))_{i=1}^p.
  \]
 Clearly, if $s_1\cdots s_p>t_1\cdots t_p^+$, then by Definition \ref{def:generalized Thue-Morse sequence} it yields $ (\th_i(s_1\cdots s_q^+))\ge(\th_i(t_1\cdots t_p^+))$.  Now we assume $s_1\cdots s_p=t_1\cdots t_p^+$, and split the proof   into the following three cases.

Case I. $p\le q< 2p$.  Then by the admissibility of $s_1\cdots s_q=s_1\cdots s_{p+j}$ it follows that
\begin{eqnarray*}
 (\th_i(s_1\cdots s_q^+))&=&s_1\cdots s_p s_{p+1}\cdots s_{p+j}^+\cdots\\
&>& s_1\cdots s_p\overline{s_1\cdots s_j}(N-1)^\f\\
&=&t_1\cdots t_p^+\overline{t_1\cdots t_j}(N-1)^\f\ge(\th_i(t_1\cdots t_p^+)).
 \end{eqnarray*}

 Case II. $q=2p$. Again by the admissibility of $s_1\cdots s_q$ we have
 \[
 s_1\cdots s_q^+=s_1\cdots s_{2p}^+\ge s_1\cdots s_p\overline{s_1\cdots s_p}^+=t_1\cdots t_p^+\overline{t_1\cdots t_p}.
 \]
 This implies that $(\th_i(s_1\cdots s_q^+))\ge(\th_i(t_1\cdots t_p^+\overline{t_1\cdots t_p}))=(\th_i(t_1\cdots t_p^+)).$

 Case III. $q>2p$. Then by the admissibility of $s_1\cdots s_q$ it follows that
 \begin{equation}\label{eq: sec6-6}
 s_{p+1}\cdots s_{2p}\ge \overline{s_1\cdots s_p}=\overline{t_1\cdots t_p^+}.
 \end{equation}
 We claim that the inequality in (\ref{eq: sec6-6}) is strict. Otherwise,
 by the admissibility of $s_1\cdots s_q$ we have
\[
   \overline{s_1\cdots s_p}s_{2p+1}\cdots s_{3p}=s_{p+1}\cdots s_{3p}\ge\overline{s_1\cdots s_{2p}}=\overline{s_1\cdots s_p}s_1\cdots s_p,
   \]
   and $   s_{2p+1}\cdots s_{3p}\le s_1\cdots s_p.$
 This implies that $s_{2p+1}\cdots s_{3p}=s_1\cdots s_p$.
 By iteration, we have
  for $q=2kp+\ell$ with $0<\ell\le 2p$,
 \[
 s_1\cdots s_q=(s_1\cdots s_{p}\overline{s_1\cdots s_p})^k s_1\cdots s_\ell.
 \]
 Then,
$
 s_{q-\ell+1}\cdots s_q=s_1\cdots s_\ell,
 $
 leading to a contradiction with the admissibility of $s_1\cdots s_q$. So, the inequality in (\ref{eq: sec6-6}) is strict, i.e.,
 \[
  s_1\cdots s_{2p}\ge s_1\cdots s_p \overline{s_1\cdots s_p}\,^+= t_1\cdots t_p^+ \overline{t_1\cdots t_p}=(\th_i(t_1\cdots t_p^+))_{i=1}^{2p}.
 \]

  Then, by induction,  it follows that  $s_1\cdots s_{2^n p}\ge (\th_i(t_1\cdots t_p^+))_{i=1}^{2^n p}$. Again by the same argument as in Case I we can show that
 $  (\th_i(s_1\cdots s_q^+))\ge(\th_i(t_1\cdots t_p^+))$.

\end{proof}

\begin{lemma}\label{prop:6-2}
Let $s_1\cdots s_q$ and $t_1\cdots t_p$ be two admissible blocks with $q\ge p$. If $(s_1\cdots s_q)^\f>(t_1\cdots t_p)^\f$, then
 $
 (\th_i(s_1\cdots s_q^+))\ge(\th_i(t_1\cdots t_p^+)).
 $
\end{lemma}
\begin{proof}
 Let $q=np+j$ with $n\ge 1$ and $0\le j<p$.
Suppose
\begin{equation}\label{eq: sec6-4}
(s_1\cdots s_q)^\f>(t_1\cdots t_p)^\f.
\end{equation}Then
$s_1\cdots s_p\ge t_1\cdots t_p$. By Lemma \ref{lem: sec6-3} it suffices to show that $s_1\cdots s_p\ne t_1\cdots t_p$.

Suppose $s_1\cdots s_p=t_1\cdots t_p$. Then by (\ref{eq: sec6-4}) and the admissibility of $s_1\cdots s_q$ it gives that
\[
s_1\cdots s_p\ge s_{p+1}\cdots s_{2p}\ge t_1\cdots t_p=s_1\cdots s_p.
\]
Then $s_1\cdots s_{2p}=(t_1\cdots t_p)^2$. By iteration, we have
\begin{equation}\label{eq: sec6-4'}
s_1\cdots s_{np}=(t_1\cdots t_p)^n.
\end{equation}
If $j=0$, i.e., $q=np$, then (\ref{eq: sec6-4'}) violates   (\ref{eq: sec6-4}). If $0<j<p$, then (\ref{eq: sec6-4'}) also leads to a contradiction, since by (\ref{eq: sec6-4}) and the admissibility of
$s_1\cdots s_q$ it follows that
\[
s_1\cdots s_j>s_{np+1}\cdots s_{np+j}\ge t_1\cdots t_j=s_1\cdots s_j.
\]
\end{proof}

\begin{proof}[Proof of Theorem \ref{th: main results-1}]
By Proposition \ref{prop: maximal admissible interval-cover} it suffices to show that either $[\al_L,\al_U]\cap[\beta_L,\beta_U]=\emptyset$ or $\al_U=\beta_U$.
By symmetry it suffices to show that ${\al_L}\in[{{\beta}_L},{{\beta}_U}]$ implies $\al_U={\beta}_U$. This can be  verified by the following observations.
  By Lemmas \ref{lem: sect6-1},  \ref{prop:6-1} and Proposition \ref{prop:1} it follows that
\[
  {\al_L}<{{\beta}_U}\quad\Longrightarrow\quad{\al_U}\le{{\beta}_U}.
  \]
  Moreover, by Lemmas \ref{lem: sect6-2}, \ref{prop:6-2} and Proposition \ref{prop:1} it follows that
  \[
  {\al_L}\ge{{\beta}_L}\quad\Longrightarrow\quad{\al_U}\ge{{\beta}_U}.
  \]
\end{proof}

\section{Proof of Theorem \ref{th: main results}}\label{sec: proof of th3}
Let
$[{{\beta}_L}, {{\beta}_U}]\subseteq[G_N, N)$  be an admissible interval generated by $t_1\cdots t_p$, i.e.,
\[
(\alpha_i({{\beta}_L}))=(t_1 \cdots t_p)^\f\quad\textrm{ and}\quad (\alpha_i({{\beta}_U}))=(\theta_i(t_1\cdots t_p^+)).
\]
By using Lemma \ref{lem: keylem} one can easily get the following lemma.
\begin{lemma}\label{lem: proof-1}
  Let $t_1\cdots t_p$ be an admissible block and  let $(\theta_i)=(\theta_i(t_1\cdots t_p^+))$. Then
  for  any $n\ge 0$,
\[
  \si^i((\theta_1\cdots \theta_{2^n p}\,\overline{\theta_1\cdots \theta_{2^n p}})^\f)\le (\theta_1\cdots \theta_{2^n p}\,\overline{\theta_1\cdots \theta_{2^n p}})^\f
  \]
  for any $i\ge 1$,   where $\si$ is the left shift  such that $\si((a_i))=(a_{i+1})$.
\end{lemma}
 By Lemma \ref{lem: proof-1} and Proposition \ref{prop:1} it follows that $(\th_1\cdots\th_{2^{n-1} p}\,\overline{\th_1\cdots \th_{2^{n-1} p}})^\f$ is the quasi-greedy expansion of $1$ for some base ${{\beta}}_n\in(1,N]$, i.e.,
 \begin{equation}\label{eq: beta_n}
 (\alpha_i({{\beta}}_n))=(\th_1\cdots\th_{2^{n-1} p}\,\overline{\th_1\cdots \th_{2^{n-1} p}})^\f.
\end{equation}
 Clearly, $( \alpha_i({{\beta}}_1))=(\th_1\cdots\th_p\,\overline{\th_1\cdots \th_p})^\f=(t_1\cdots t_p^+\,\overline{t_1\cdots t_p^+})^\f$, and the first $2^{n-1}p$ elements of $(\alpha_i({{\beta}_{n}}))$ coincide with that of the generalized Thue-Morse sequence  $(\theta_i(t_1\cdots t_p^+))$. So, as $n\ra \f$ the sequence $(\alpha_i({{\beta}}_n))$ increasingly converges to the generalized Thue-Morse sequence $(\theta_i(t_1\cdots t_p^+))$. By Proposition \ref{prop:1} it gives that  ${{\beta}}_n$ converges to ${{\beta}_U}$ from the left.

Recall from  Theorem \ref{th: 3} that $\V{N}$ is defined by
\[
 \V{N}=\Big\{\sum_{i=1}^\f\frac{d_i}{{\beta}^i}: ~\overline{\alpha_1\alpha_2\cdots}<d_n d_{n+1}\cdots<\alpha_1\alpha_2\cdots, ~n\ge 1\Big\}.
\]
 The following lemma
investigates all possible  blocks occuring in the ${\beta}$-expansions of
points in $\V{N}$ for ${\beta}\le{{\beta}}_1$.
\begin{lemma}\label{lem:1}
 Let $t_1\cdots t_p$ be an admissible block and  let $(\alpha_i({{\beta}}_1))=(t_1\cdots t_p^+\overline{t_1\cdots t_p^+})^\f$. If ${\beta}\le {{\beta}}_1$, then
  $\V{N}\subseteq \Pi_{\beta}(Z_{t_1\cdots t_p})$, where
  $$
  Z_{t_1\cdots t_p}:=\big\{(d_i): \overline{t_1\cdots t_p}\le d_n\cdots d_{n+p-1}\le t_1\cdots t_p, n\ge 1\big\}.
  $$
\end{lemma}
\begin{proof}
Since ${\beta}\le{{\beta}}_1$,  it follows from Proposition \ref{prop:1}  that
$
(\alpha_i({\beta}))\le(\alpha_i({{\beta}}_1))=(t_1\cdots t_p^+\overline{t_1\cdots t_p^+})^\f.
$
Take $x=\Pi_{\beta}((d_i))\in\V{N}$. Then for all $n\ge 1$,
{\small\begin{equation}\label{eq: proof-0}
(\overline{t_1\cdots t_p^+}t_1\cdots t_p^+)^\f\le(\overline{\alpha_i({{\beta}})})< d_n d_{n+1}\cdots< (\alpha_i({\beta}))\le(t_1\cdots t_p^+\overline{t_1\cdots t_p^+})^\f.
\end{equation}}
This implies
\begin{equation*}\label{eq: proof-1}
\overline{t_1\cdots t_p^+}\le d_n d_{n+1}\cdots d_{n+p-1}\le t_1\cdots t_p^+.
\end{equation*}
We will finish the proof by showing that the inequalities in the above equation are strict.

Suppose $d_n d_{n+1}\cdots d_{n+p-1}=t_1\cdots t_p^+$. Then by Equation (\ref{eq: proof-0})  it follows that
$
 d_{n+p}d_{n+p+1}\cdots d_{n+2p-1}\le\overline{t_1\cdots t_p^+}.
$
Again  by Equation (\ref{eq: proof-0}) we have
$
d_{n+p}d_{n+p+1}\cdots d_{n+2p-1}\ge\overline{t_1\cdots t_p^+}.
$
Then
\[
d_{n+p} d_{n+p+1}\cdots d_{n+2p-1}=\overline{t_1\cdots t_p^+}.
\]
 By iteration, we have
$
d_n d_{n+1}\cdots=(t_1\cdots t_p^+\overline{t_1\cdots t_p^+})^\f,
$
leading to a contradiction with (\ref{eq: proof-0}).

Similarly, one can show that $d_n d_{n+1}\cdots d_{n+p-1}\ne \overline{t_1\cdots t_p^+}$.
\end{proof}
By Lemma \ref{lem:1} it yields that $\dim_H\V{N}\le \dim_H\Pi_{\beta}(Z_{t_1\cdots t_p})$ for ${\beta}\le{{\beta}}_1$. In the following lemma we will  show that
$\dim_H\V{N}\ge \dim_H\Pi_{\beta}(Z_{t_1\cdots t_p})$ for ${\beta}\ge {{\beta}_L}$.

\begin{lemma}\label{lem:2}
Let $t_1\cdots t_p$ be an admissible block and let $(\alpha_i({{\beta}_L}))=(t_1\cdots t_p)^\f$. If ${\beta}\ge {{\beta}_L}$, then
  $\dim_H\V{N}\ge\dim_H\Pi_{\beta}(Z_{t_1\cdots t_p})$.
\end{lemma}
\begin{proof}
 By the definition of $Z_{t_1\cdots t_p}$ it follows that $(\overline{t_1\cdots t_p})^\f$ and $(t_1\cdots t_p)^\f$ are the least and the largest elements in $Z_{t_1\cdots t_p}$, respectively. Accordingly,
 let $t_*$ and $t^*$ be respectively the least and the largest
  elements in $\Pi_{\beta}(Z_{t_1\cdots t_p})$, i.e.,
{\small\[
  t_*=\Pi_{{\beta}}((\overline{t_1\cdots t_p})^\f)=\frac{\sum_{i=1}^p \overline{t_i}\,{{\beta}}^{p-i}}{{{\beta}}^p-1}, \quad
  t^*=\Pi_{{\beta}}((t_1\cdots t_p)^\f)=\frac{\sum_{i=1}^p {t_i}{{\beta}}^{p-i}}{{{\beta}}^p-1}.
\]}
  Set
{\small\[
T=\bigcup_{n\ge 0}\Big(\big\{\sum_{i=1}^n
\frac{d_i}{{{\beta}}^i}+ \frac{t_*}{{{\beta}}^n}: 0\le d_i\le
N-1\big\}\cup\big\{\sum_{i=1}^n \frac{d_i}{{{\beta}}^i}+ \frac{t^*}{{{\beta}}^n}:
0\le d_i\le N-1\big\}\Big).
  \]}  Clearly, $T$ is a countable set. Then it suffices to show that $\Pi_{{\beta}}(Z_{t_1\cdots t_p})\setminus
  T\subseteq\V{N}$.
   Take $x=\Pi_{\beta}((d_i))\in \Pi_{{\beta}}(Z_{t_1\cdots t_p})\setminus
  T$. We claim that $d_n d_{n+1}\cdots<\al_1(\beta)\al_2(\beta)\cdots$ for any $n\ge 1$.

Suppose that there exists $n_0\ge 1$ such that $d_{n_0}d_{n_0+1}\cdots\ge(\alpha_i({\beta}))$. Since ${\beta}\ge {{\beta}_L}$,  by Proposition \ref{prop:1} it follows that
\[
d_{n_0}d_{n_0+1}\cdots\ge(\alpha_i({\beta}))\ge(\alpha_i({{\beta}_L}))=(t_1\cdots t_p)^\f.
\]
Since $x\notin T$, we  have
$d_{n_0}d_{n_0+1}\cdots>(t_1\cdots t_p)^\f$. Then there exists a nonnegative integer
$s$ such that
$
d_{n_0}d_{n_0+1}\cdots d_{n_0+s p-1}=(t_1\cdots t_p)^{s}
$
and
\[
d_{n_0+s p}d_{n_0+s p+1}\cdots d_{n_0+s p+p-1}>t_1\cdots
t_p,
\]
leading  to a contradiction with $x\in \Pi_{{\beta}}(Z_{t_1\cdots t_p})$.
Thus, $d_n d_{n+1}\cdots < (\alpha_i({\beta}))$ for any
$n\ge 1$.

Similarly, one can show that $d_n d_{n+1}\cdots
>(\overline{\alpha_i({\beta})})$ for any $n\ge 1$.
 So $x\in\V{N}$, and we conclude that
$\Pi_{{\beta}}(Z_{t_1\cdots t_p})\setminus
  T\subseteq\V{N}$.
\end{proof}
In the following we will  investigate  the structure of $\Pi_{\beta}(Z_{t_1\cdots t_p})$. If $p=1$, then $\Pi_\beta(Z_{t_1})$ is a self-similar set whose structure is well-studied (cf.~\cite{Hutchinson_1981}). So, we only need to consider the case for $p\ge 2$. Note that $(d_i)\in Z_{t_1\cdots t_p}$ if and only if $d_n d_{n+1}\cdots d_{n+p-1}\notin{\mathcal{F}}$ for any $n\ge 1$, where
\[
{\mathcal{F}}:=\big\{c_1\cdots c_p: c_1\cdots c_p<\overline{t_1\cdots t_p} ~\textrm{or}~c_1\cdots c_p>t_1\cdots t_p\big\}.
\]
Then $Z_{t_1\cdots t_p}$ is a $p-1$ step of shift of finite type (cf.~\cite{Lind_Marcus_1995}). We construct an edge graph  ${\mathcal{G}}=(G, {V}, {E})$ with the vertices set ${V}$  defined by
\[
{V}:=\big\{u_1\cdots u_{p-1}: \overline{t_1\cdots t_{p-1}}\le u_1\cdots u_{p-1}\le t_1\cdots t_{p-1}\big\}.
\]
For two vertices $\mathbf{u}=u_1\cdots u_{p-1}, \mathbf{v}=v_1\cdots v_{p-1}\in{V}$, we draw an edge $\mathbf{uv}\in {E}$ from $\mathbf{u}$ to $\mathbf{v}$ and label it $\ell_{\mathbf{uv}}=u_1$ if $u_2\cdots u_{p-1}=v_1\cdots v_{p-2}$\footnote{When $p=2$ this holds automatically.} and $u_1\cdots u_{p-1}v_{p-1}\notin{\mathcal{F}}$. One can check that the edge graph ${\mathcal{G}}=(G, {V}, {E})$ is a representation of $Z_{t_1\cdots t_p}$.

\begin{lemma}\label{lem:2'}
 Let $t_1\cdots t_p$ be an admissible block with $p\ge 2$ and let  $(\alpha_i({{\beta}_L}))=(t_1\cdots t_p)^\f$. Then for any ${\beta}\ge {{\beta}_L}$ the set $\Pi_{\beta}(Z_{t_1\cdots t_p})$ is a graph-directed set satisfying the SSC.
\end{lemma}
\begin{proof}
  Let ${\mathcal{G}}=(G, {V}, {E})$ be the edge graph representing  $Z_{t_1\cdots t_p}$.
  For $\mathbf{u}=u_1\cdots u_{p-1}\in{V}$, let
\[
K_{\mathbf{u}}:=\Big\{\sum_{i=1}^\f\frac{d_i}{{\beta}^i}:
d_i=u_i, 1\le i\le p-1, ~\textrm{and}~  d_{n}\cdots
d_{n+p-1}\notin{\mathcal{F}}, n\ge 1\Big\}.
\]
For an edge $\mathbf{uv}\in{E}$ with $\mathbf{u}=u_1\cdots u_{p-1}, \mathbf{v}=v_1\cdots v_{p-1}\in{V}$ we define the map $f_{\mathbf{uv}}$ as
\begin{equation}\label{eq: lem2'-1}
f_{\mathbf{uv}}(x)=\frac{x+\ell_{uv}}{{\beta}}=\frac{x+u_1}{{\beta}}.
\end{equation}
We claim  that for any $\mathbf{u}\in{V}$,
\begin{equation}\label{eq: lem2'-2}
K_{\mathbf{u}}=\bigcup_{\mathbf{uv}\in {E}}f_{\mathbf{uv}}(K_\mathbf{v}).
 \end{equation}

Take $\Pi_{\beta}((s_i))\in K_\mathbf{u}$. Then $s_1=u_1, \cdots, s_{p-1}=u_{p-1}$; and
$
\overline{t_1\cdots t_p}\le s_n\cdots s_{n+p-1}\le t_1\cdots t_p
$
 for any $n\ge 1$. This implies that
\[
\mathbf{v}:=s_2\cdots s_p=u_2\cdots u_{p-1}s_p\in{V}\quad\textrm{ and }\quad \mathbf{uv}\in{E}.
\] Then by Equation (\ref{eq: lem2'-1}) we have
\begin{eqnarray*}
&&\Pi_{\beta}((s_i))\in f_{\mathbf{uv}}(K_\mathbf{v})=\Big\{\sum_{i=1}^\f\frac{d_i}{{\beta}^i}: d_i=u_i, 1\le i\le p-1; d_p=s_p;\\
&&\hspace{4.5cm}\textrm{and}~\overline{t_1\cdots t_p}\le d_n\cdots d_{n+p-1}\le t_1\cdots t_p, n\ge 1\Big\}.
\end{eqnarray*}
So, $K_\mathbf{u}\subseteq \bigcup_{\mathbf{uv}\in {E}}f_{\mathbf{uv}}(K_\mathbf{v}).$

 For the other inclusion of Equation (\ref{eq: lem2'-2}) we take $\Pi_{\beta}((s_i))\in\bigcup_{\mathbf{uv}\in {E}}f_{\mathbf{uv}}(K_\mathbf{v})$. Then there exist
 $\mathbf{uv}\in{E}$ with $\mathbf{u}=u_1\cdots u_{p-1}, \mathbf{v}=v_1\cdots v_{p-1}\in{V}$ such that $\Pi_{\beta}((s_i))\in f_{\mathbf{uv}}(K_\mathbf{v})$. This implies that
$
s_i=u_i, 1\le i\le p-1;~s_p=v_{p-1}$ and
\[
\overline{t_1\cdots t_p}\le s_n\cdots s_{n+p-1}\le t_1\cdots t_p,~ n\ge 1.
\]
So, $\Pi_{\beta}((s_i))\in K_\mathbf{u}$ and we conclude that $\bigcup_{\mathbf{uv}\in {E}}f_{\mathbf{uv}}(K_\mathbf{v})\subseteq K_\mathbf{u}$. Then Equation (\ref{eq: lem2'-2}) holds.

Similarly, one can  check that
\[
\Pi_{\beta}(Z_{t_1\cdots t_p})=\bigcup_{\mathbf{v}\in{V}}K_\mathbf{v}.
\]
 So, $\Pi_{\beta}(Z_1\cdots t_p)$ is a graph-directed set generated by the IFS
$
\{(K_{\mathbf{u}})_{\mathbf{u}\in{V}}, (f_{\mathbf{uv}})_{\mathbf{uv}\in{E}}\}$ (cf.~\cite{Mauldin_Williams_1988}).
We will finish the proof by showing that the IFS $\{(K_{\mathbf{u}})_{\mathbf{u}\in{V}}, (f_{\mathbf{uv}})_{\mathbf{uv}\in{E}}\}$  satisfies the SSC.

Since ${\beta}\ge{{\beta}_L}$, it follows from the proof of Lemma \ref{lem:2} that for any $(d_i)\in Z_{t_1\cdots t_p}$ we have
\[
\overline{\alpha_1({\beta})\alpha_2({\beta})\cdots}\le d_n d_{n+1}\cdots\le\alpha_1({\beta})\alpha_2({\beta})\cdots\quad\textrm{for any}~n\ge 1.
\]
By Proposition \ref{prop:1} this implies that $\Pi_{\beta}(Z_{t_1\cdots t_p})\subseteq[0,1]$. Let $\mathbf{uv}, \mathbf{uv}'\in{E}$ with $\mathbf{u}=u_1\cdots u_{p-1}, \mathbf{v}=v_1\cdots v_{p-1}$ and $\mathbf{v}'=v_1'\cdots v_{p-1}'$. Suppose $v_{p-1}<v'_{p-1}$. Then
\begin{eqnarray*}
\sum_{i=1}^{p-1}\frac{u_i}{{\beta}^i}+\frac{v_{p-1}}{{\beta}^p}+\sum_{i=1}^\f\frac{d_i}{{\beta}^{p+i}}&\le&\sum_{i=1}^{p-1}\frac{u_i}{{\beta}^i}+\frac{v_{p-1}+1}{{\beta}^p}\\
&<&\sum_{i=1}^{p-1}\frac{u_i}{{\beta}^i}+\frac{v'_{p-1}}{{\beta}^p}+\sum_{i=1}^\f\frac{d_i'}{{\beta}^{p+i}}
\end{eqnarray*}
for any $(d_i), (d_i')\in Z_{t_1\cdots t_p}$.   This yields  $f_{\mathbf{uv}}(K_\mathbf{v})\cap f_{\mathbf{uv}'}(K_{\mathbf{v}'})=\emptyset$.
\end{proof}

When $p=1$ one can easily get the following lemma.
\begin{lemma}
  \label{lem:2''}
Let $t_1$ be an admissible block and let $(\alpha _i(\beta _L))=t_1^\infty$. Then for any $\beta \geq \beta _L$ the set $\Pi _\beta (Z_{t_1})$ is a self similar set satisfying SSC.
\end{lemma}
Now we give the Hausdorff dimension of $\U{N}$ for $\beta\in[\beta_L, \beta_1]$.
\begin{proposition}\label{lem:3}
  Let  $t_1\cdots t_p$ be an admissible block and let $(\alpha_i({{\beta}_L}))=(t_1\cdots t_p)^\f, (\alpha_i({{\beta}}_1))=(t_1\cdots t_p^+\overline{t_1\cdots t_p^+})^\f$. Then for any ${\beta}\in[{{\beta}_L}, {{\beta}}_1]$ the Hausdorff dimension of $\U{N}$ is given by
\[
 \dim_H\U{N}=\frac{h(Z_{t_1\cdots t_p})}{\log{\beta}},
 \]
  where $h(Z_{t_1\cdots t_p})$ is the topological entropy of the subshift of finite type
 $Z_{t_1\cdots t_p}$.
\end{proposition}
\begin{proof}
By Lemmas   \ref{lem:1},  \ref{lem:2} and Theorem \ref{th: 3} it follows  that for any ${\beta}\in[{{\beta}_L},{{\beta}}_1]$,
\[
\dim_H\U{N}=\dim_H \Pi_{\beta}(Z_{t_1\cdots t_p}).
\]
 By Lemma \ref{lem:2'} and Lemma  \ref{lem:2''} $\Pi_{\beta}(Z_1\cdots t_p)$ is a graph-directed set or a self-similar set satisfying the SSC. Then the Hausdorff dimension of $\Pi_{\beta}(Z_{t_1\cdots t_p})$ can be calculated via the topological entropy of $Z_{t_1\cdots t_p}$ (cf.~\cite{Lind_Marcus_1995}), i.e.,
$
\dim_H\Pi_{\beta}(Z_{t_1\cdots t_p})={h(Z_{t_1\cdots t_p})}/{\log{\beta}}.
$
\end{proof}

\begin{proof}[Proof of Theorem \ref{th: main results}]
Recall from (\ref{eq: beta_n}) that ${{\beta}}_n$ is defined by
\[
(\alpha_i({{\beta}_n}))=(\th_1\cdots\th_{2^{n-1}p}\,\overline{\th_1\cdots\th_{2^{n-1}p}})^\f.
\]
Note by Lemma \ref{lem: keylem} that $t_1\cdots t_p^+ \overline{t_1\cdots t_p^+}$ is admissible. Then by Proposition \ref{lem:3} it follows that for any ${\beta}\in[{{\beta}}_1,{{\beta}_2}]$
\begin{equation*}\label{eq: sec4-2}
\dim_H \U{N}=\frac{h(Z_{t_1\cdots t_p^+\overline{t_1\cdots t_p^+}})}{\log{\beta}}.
\end{equation*}
By taking  ${\beta}={{\beta}}_1$ in the above equation and in Proposition \ref{lem:3}  it follows that
$h(Z_{t_1\cdots t_p})=h(Z_{t_1\cdots t_p^+\overline{t_1\cdots t_p^+}}).$
 So, for any ${\beta}\in[{{\beta}_L},{{\beta}_2}]$ we  have
$
\dim_H\mathbf{U}_{{\beta}, N}={h(Z_{t_1\cdots t_p})}/{\log{{\beta}}}.
$
By induction, we have
\[
\dim_H\U{N}=\frac{h(Z_{t_1\cdots t_p})}{\log{{\beta}}}
 \]
 for any ${\beta}\in[{{\beta}_L},{{\beta}_n}]$. Letting $n\ra\f$ we have by Proposition \ref{prop:1} that ${{\beta}_n}\ra{{\beta}_U}$.  The authors in \cite{Komornik_Kong_Li_2014}  showed that the map ${\beta}\ra\dim_H\U{N}$ is continuous for ${\beta}>1$. This establishes Theorem \ref{th: main results}.
\end{proof}

\begin{remark}
Let $\overline{\U{N}}$ denote the closure of $\U{N}$. The authors in \cite{Komornik_Kong_Li_2014} showed  for $\beta>1$
 that the set $\U{N}$ may be not closed, and the set $\overline{\U{N}}\setminus\U{N}$ is at most countable. Then for ${\beta}\in[{{\beta}_L}, {{\beta}_U}]$,
\[
\dim_H\overline{\U{N}}=\dim_H\U{N}=\frac{h(Z_{t_1\cdots t_p})}{\log{\beta}}.
\]
\end{remark}

\section{Explicit formulae for the Hausdorff dimensions of $\U{N}$}\label{sec: examples}

In this section we consider some examples for which the Hausdorff dimension of $\U{N}$ can be calculated explicitly. An admissible interval $[{{\beta}_L}, {{\beta}_U}]$ is called a \emph{$p$-level admissible interval} if $[{{\beta}_L}, {{\beta}_U}]$ can be generated by an admissible block $t_1\cdots t_p$ of length $p$. First we will consider the case for the $1$-level admissible intervals.

\begin{theorem}\label{th: leve1 one}
   Given $N\ge 3$, let $[{{\beta}_L}, {{\beta}_U}]$ be an admissible interval generated by an admissible block $t_1\in\{0,1,\cdots,N-1\}$. Then $\lceil(N-1)/2\rceil\le t_1\le N-2$, and  for any ${\beta}\in[{{\beta}_L}, {{\beta}_U}]$ the Hausdorff dimension of $\U{N}$ is given by
\[
  \dim_H\U{N}=\frac{\log(2t_1+2-N)}{\log{\beta}}.
  \]
\end{theorem}
\begin{proof}
By Definition \ref{def:admissible interval} it follows that
$
(\alpha_i({{\beta}_L}))=t_1^\f\quad\textrm{and}\quad (\alpha_i({{\beta}_U}))=(\th_i(t_1+1)).
$
 Since $t_1\in\{0,1,\cdots,N-1\}$ is an admissible block, by Definition \ref{def:admissible block} it gives that
$\lceil(N-1)/2\rceil\le t_1\le N-2$. By Theorem \ref{th: main results} it follows that for any ${\beta}\in[{{\beta}_L}, {{\beta}_U}]$ the Hausdorff dimension of $\U{N}$ is given by
\[
\dim_H\U{N}=\frac{h(Z_{t_1})}{\log{\beta}},
\]
where
$Z_{t_1}=\{(d_i): \overline{{t_1}}\le d_n\le {t_1}, n\ge 1\}.$
So, the theorem follows by an easy calculation that $h(Z_{t_1})=\log(t_1-\overline{{t_1}}+1)=\log(2{t_1}+2-N).$
\end{proof}
If we take  ${t_1}=N-2$ in Theorem \ref{th: leve1 one}, then we extend the main result of  Kall\'{o}s \cite{Kallos_1999}. This can be seen by the following observation.
Clearly, ${{\beta}_L}=N-1$. By Definition \ref{def:generalized Thue-Morse sequence} of the generalized Thue-Morse sequence $(\th_i(N-1))$ it follows that
\begin{eqnarray*}
(\alpha_i({{\beta}_U}))=(\th_i(N-1))&=&(N-1)1\, 0(N-1)\,0(N-2)\,(N-1)1\cdots\\
&>&((N-1)0)^\f=\Big(\alpha_i\big(\frac{N-1+\sqrt{N^2-2N+5}}{2}\big)\Big).
\end{eqnarray*}
By Proposition \ref{prop:1} this implies that ${{\beta}_U}>(N-1+\sqrt{N^2-2N+5})/2$.

%
Now we consider the Hausdorff dimension of $\U{N}$ for ${\beta}$ in any $2$-level admissible intervals.
\begin{theorem}\label{th: leve two}
 Given $N\ge 2$, let $[{{\beta}_L}, {{\beta}_U}]$ be an admissible interval generated by an admissible block $t_1 t_2$. Then $\lceil(N-1)/2\rceil\le t_1\le N-1$, $\overline{t_1}\le t_2<t_1$, and for any ${\beta}\in[{{\beta}_L}, {{\beta}_U}]$ the Hausdorff dimension of $\U{N}$ is given by
{\small\[
  \dim_H\mathbf{U}_{{{\beta}},N}=\frac{\log (2{{t}_1}+1-N+\sqrt{(2{{t}_1}+1-N)^2+4(2{t_2}+2-N)}\;)-\log 2}{\log{{\beta}}}.
\]}
 \end{theorem}
\begin{proof}
Since $t_1t_2$ is an admissible block, by Definition \ref{def:admissible block} it follows that
\[
\overline{t_1}\le t_1\le N-1\quad\textrm{and} \quad \overline{t_1}\le t_2<t_1.
\]
By Theorem \ref{th: main results}  it suffices to calculate the entropy of $Z_{t_1t_2}$.

 Let $\mathcal{G}=\{G, {V},{E}\}$ be an edge graph representing the shift of finite type $Z_{t_1t_2}$, where the vertex set
  ${V}=\{\overline{t_1},\overline{t_1}+1,\cdots, t_1\}$ and the edge set ${E}$ consists of all edges $uv$ satisfying  $\overline{t_1 t_2}\le uv\le t_1 t_2$ for $u, v\in {V}$. Note that the entropy of $Z_{t_1 t_2}$ can be calculated via the spectral radius of the adjacency matrix $A$ of the edge graph $\mathcal{G}$ (cf.~\cite{Lind_Marcus_1995}), where $A$ is of size $(t_1-\overline{t_1}+1)\times(t_1-\overline{t_1}+1)$
 given by
\begin{equation*}
A=\left(\begin{array}{ccccccc}
  0&0&\cdots&0&1&\cdots&1\\
  1&1&1&\cdots&\cdots&\cdots&1\\
  1&1&1&1&\cdots&\cdots&1\\
  \vdots&&&\ddots&&&\vdots\\
  1&\cdots&\cdots&1&1&1&1\\
  1&\dots&\dots&\dots&1&1&1\\
  1&\cdots&1&0&\cdots&0&0
\end{array}\right).
\end{equation*}
Here the  total number of zeros on the top and the bottom rows are both equal to $t_1-\overline{t_2}+1=t_2-\overline{t_1}+1$.
Then
\[
h(Z_{t_1t_2})=\log\frac{(2{t_1}+1-N)+\sqrt{(2{t_1}+1-N)^2+4(2{t_2}+2-N)}}{2}.
\]
This completes the proof.
\end{proof}
The authors in \cite{Glendinning_Sidorov_2001,Kong_Li_Dekking_2010} showed that $\dim_H\U{N}=0$ when ${\beta}={\beta}_c(N)$. This can also be viewed  by
Theorem \ref{th: leve1 one} and \ref{th: leve two}.
\begin{corollary}\label{cor:1}
  Given  $N\ge 2$, for any ${\beta}\in[G_N, {\beta}_c(N)]$ we have $\dim_H\U{N}=0$.
\end{corollary}
\begin{proof}
We split the proof  into the following two cases.

Case I. $N=2k$. By Equations (\ref{eq: gn}) and (\ref{eq: lambda}) it follows that
\[
(\alpha_i(G_N))=(k(k-1))^\f\quad\textrm{and}\quad(\alpha_i({\beta}_c(N)))=(\th_i(kk)).
\]
So, $[G_N, {\beta}_c(N)]$ is an admissible interval generated by the admissible block $k(k-1)$. By Theorem \ref{th: leve two} it follows that for  ${\beta}= {\beta}_c(N)$ the set $\U{N}$ has zero Hausdorff dimension.

Case II. $N=2k+1$. By Equations (\ref{eq: gn}) and (\ref{eq: lambda}) one can check that
$[G_N, {\beta}_c(N)]$ is an admissible interval generated by the admissible  block $k$. Then  by Theorem \ref{th: leve1 one} it follows that  for  ${\beta}= {\beta}_c(N)$ we have $\dim_H\U{N}=0$.
\end{proof}

\begin{example}\label{ex£º2}
  Let $N=20$. According to Theorem \ref{th: leve1 one} and Theorem \ref{th: leve two}, we plot in Figure \ref{fig:2} the graph of the Hausdorff dimension
  $\dim_H\U{20}$ of $\U{20}$ for ${\beta}\in({\beta}_c(20), 20)$.  Clearly,  the 1-level and 2-level admissible intervals cover a large part of $[{\beta}_c(N),N)$. By Theorem \ref{th: main results-1} the union of all admissible intervals covers almost every point of $(\beta_c(N),N)$. Thus, the dimension function $\dim_H\U{N}$ has a devil's-staircase-like behavior.
  \begin{figure}[h!]
  {\centering
  \includegraphics[width=6cm]{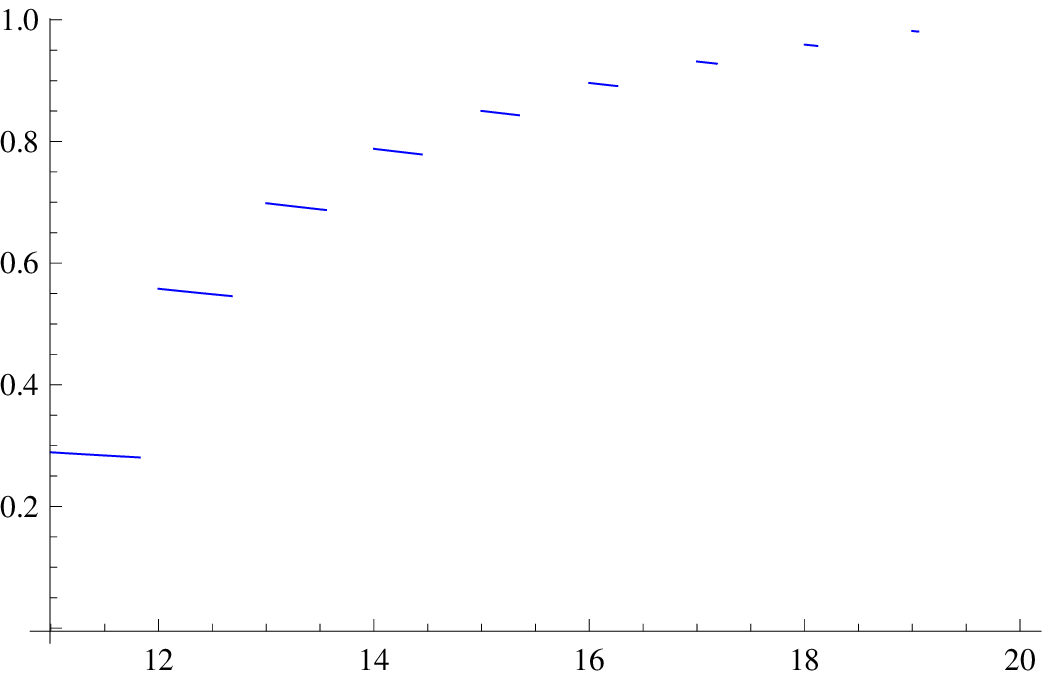}\quad\includegraphics[width=6cm]{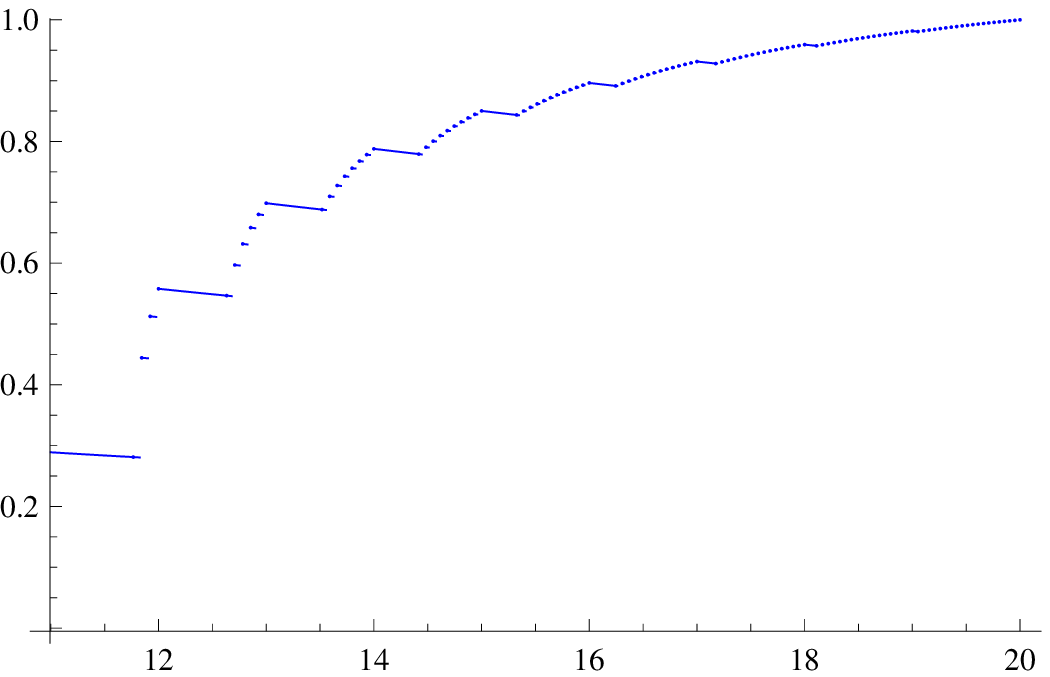}\\
  \caption{The Hausdorff dimension of $\U{20}$ for ${\beta}\in({\beta}_c(20), 20)$. In the left column ${\beta}$ is in the 1-level admissible intervals; In the right column ${\beta}$ is in the 1-level and 2-level admissible intervals.}\label{fig:2}
  }
\end{figure}
\end{example}

\section*{Acknowledgement}
The authors thank the anonymous referees for many suggestions and remarks.
The first author would like to thank  Vilmos Komornik and Michel Dekking for some suggestions on the previous versions of the manuscript, and to G\'{a}bor Kall\'{o}s and Simon Baker for  providing
some useful references. In particular, we thank Martijn De Vries for correcting an error we made in Theorem 2.5 in the original version. The first author is supported by the National  Natural Science Foundation of China no 11326207, 11401516 and JiangSu Province Natural Science Foundation for the Youth no BK20130433. The second author is
supported by the National  Natural Science Foundation of China no 11271137.


\begin{thebibliography}{10}

\bibitem{Allouche_Cosnard_1983}
Allouche J and Cosnard M 1983
\newblock It\'erations de fonctions unimodales et suites engendr\'ees par
  automates
\newblock {\em C. R. Acad. Sci. Paris S\'er. I Math.} \textbf{296} 159--162


\bibitem{Allouche_Cosnard_2000}
Allouche J and  Cosnard M 2000
\newblock The {K}omornik-{L}oreti constant is transcendental
\newblock {\em Amer. Math. Monthly} \textbf{107} 448--449


\bibitem{Allouche_Frougny_2009}
 Allouche J and Frougny C 2009
\newblock Univoque numbers and an avatar of {T}hue-{M}orse
\newblock {\em Acta Arith.}  \textbf{136} 319--329

\bibitem{Allouche_Shallit_1999}
 Allouche J and  Shallit J 1999
\newblock The ubiquitous {P}rouhet-{T}hue-{M}orse sequence
\newblock {\em In Sequences and their applications ({S}ingapore, 1998)},
  Springer Ser. Discrete Math. Theor. Comput. Sci.    1--16

\bibitem{Baatz_Komornik_2011}
Baatz  M and  Komornik V 2011
\newblock Unique expansions in integer bases with extended alphabets
\newblock {\em Publ. Math. Debrecen}  \textbf{79} 251--267

\bibitem{Baiocchi_Komornik_2007}
Baiocchi C and  Komornik V 2007
\newblock Greedy and quasi-greedy expansions in non-integer bases
\newblock {\em arXiv:0710.3001v1}

\bibitem{Baker_2012}
Baker S 2012
\newblock Generalised golden ratios over integer alphabets
\newblock {\em Preprint:arXiv:1210.8397v1}

\bibitem{Barrera_2013}
Barrera R 2014
\newblock Topological and ergodic properties of symmetric sub-shifts
\newblock {\em Discrete Contin. Dyn. Syst.}
\textbf{34} 4459--4486

\bibitem{Dajani_DeVries_2007}
 Dajani K and  de~Vries M 2007
\newblock Invariant densities for random {$\beta$}-expansions
\newblock {\em J. Eur. Math. Soc. (JEMS)} \textbf{ 9 } 157--176

\bibitem{Dajani_Kraaikamp_2003}
Dajani K and  Kraaikamp C 2003
\newblock Random {$\beta$}-expansions
\newblock {\em Ergodic Theory Dynam. Systems}  \textbf{23}  461--479

\bibitem{Daroczy_Katai_1993}
Dar{\'o}czy Z and  K{\'a}tai I 1993
\newblock Univoque sequences
\newblock {\em Publ. Math. Debrecen}   \textbf{42}  397--407

\bibitem{Darczy_Katai_1995}
Dar{\'o}czy Z and K{\'a}tai I 1995
\newblock On the structure of univoque numbers
\newblock {\em Publ. Math. Debrecen}   \textbf{46 } 385--408

\bibitem{DeVries_Komornik_2008}
de~Vries M and  Komornik V 2009
\newblock Unique expansions of real numbers.
\newblock {\em Adv. Math.}   \textbf{221}  390--427

\bibitem{DeVries_Komornik_2011}
 de~Vries M and  Komornik V 2011
\newblock A two-dimensional univoque set
\newblock {\em Fund. Math.}   \textbf{212}  175--189

\bibitem{Erdos_Joo_Komornik_1990}
Erd\"{o}s P,  Jo\'{o} I and  Komornik V 1990
\newblock Characterization of the unique expansions $1=\sum_{i=1}^\infty
  q^{-n_i}$ and related problems
\newblock {\em Bull. Soc. Math. France}  \textbf{118}  377--390

\bibitem{Falconer_1990}
Falconer K 1990
\newblock {\em Fractal geometry--Mathematical foundations and applications }
\newblock John Wiley \& Sons Ltd.  Chichester


\bibitem{Frougny_Solomyak_1992}
Frougny C and  Solomyak B 1992
\newblock Finite beta-expansions
\newblock {\em Ergodic Theory Dynam. Systems}  \textbf{12}  713--723

\bibitem{Glendinning_Sidorov_2001}
Glendinning P and  Sidorov N 2001
\newblock Unique representations of real numbers in non-integer bases
\newblock {\em Math. Res. Lett.}  \textbf{8}  535--543

\bibitem{Hutchinson_1981}
Hutchinson J 1981
\newblock Fractals and self-similarity
\newblock {\em Indiana Univ. Math. J.}   \textbf{30}  713--747

\bibitem{Kallos_1999}
Kall{\'o}s G 1999
\newblock The structure of the univoque set in the small case
\newblock {\em Publ. Math. Debrecen}  \textbf{54}  153--164

\bibitem{Kallos_2001}
Kall{\'o}s G 2001
\newblock The structure of the univoque set in the big case
\newblock {\em Publ. Math. Debrecen} \textbf{ 59}  471--489

\bibitem{Komornik_Kong_Li_2014}
Komornik V,  Kong D R and  Li W X 2014
\newblock On the hausdorff dimension of univoque sets
\newblock {\em Preprint}

\bibitem{Komornik_Loreti_2002}
 Komornik  V and  Loreti P 2002
\newblock Subexpansions, superexpansions and uniqueness properties in
  non-integer bases
\newblock {\em Period. Math. Hungar.}   \textbf{44}  197--218

\bibitem{Komornik_Loreti_2007}
Komornik V and  Loreti P 2007
\newblock On the topological structure of univoque sets
\newblock {\em J. Number Theory}  \textbf{122}  157--183


\bibitem{Kong_Li_Dekking_2010}
Kong D R,  Li W X and  Dekking M 2010
\newblock Intersections of homogeneous {C}antor sets and beta-expansions
\newblock {\em Nonlinearity}   \textbf{23}  2815--2834

\bibitem{Lind_Marcus_1995}
Lind D and  Marcus B 1995
\newblock {\em An introduction to symbolic dynamics and coding}
\newblock Cambridge University Press, Cambridge

\bibitem{Mahler_1976}
Mahler K 1976
\newblock {\em Lectures on transcendental numbers}
\newblock Lecture Notes in Mathematics, Vol. 546. Springer-Verlag, Berlin

\bibitem{Mauldin_Williams_1988}
Mauldin D and  Williams C 1988
\newblock Hausdorff dimension in graph directed constructions
\newblock {\em Trans. Amer. Math. Soc.}  \textbf{309}  811--829

\bibitem{Parry_1960}
Parry W 1960
\newblock On the $\beta$-expansions of real numbers
\newblock {\em Acta Math. Acad. Sci. Hungar.}  \textbf{11 } 401--416

\bibitem{Petho_Tichy_1989}
 Peth\"{o} A and  Tichy R 1989
\newblock On digit expansions with respect to linear recurrences
\newblock {\em J. Number Theory}   \textbf{33}   243--256

\bibitem{Renyi_1957}
R\'{e}nyi A 1957
\newblock Representations for real numbers and their ergodic properties
\newblock {\em Acta Math. Acad. Sci. Hungar.}  \textbf{ 8}  477--493

\bibitem{Schmidt_1980}
Schmidt K 1980
\newblock On periodic expansions of {P}isot numbers and {S}alem numbers
\newblock {\em Bull. London Math. Soc.} \textbf{ 12}  269--278

\bibitem{Sidorov_2003}
Sidorov N 2003
\newblock Almost every number has a continuum of {$\beta$}-expansions
\newblock {\em Amer. Math. Monthly}   \textbf{110}  838--842

\bibitem{Sidorov_2003-1}
Sidorov N 2003
\newblock Universal {$\beta$}-expansions
\newblock {\em Period. Math. Hungar.}   \textbf{47}  221--231

\bibitem{Sidorov_2007}
Sidorov N 2007
\newblock Combinatorics of linear iterated function systems with overlaps
\newblock {\em Nonlinearity}   \textbf{20}  1299--1312

\bibitem{Tan_Wang_2011}
Tan  B and  Wang B W 2011
\newblock Quantitative recurrence properties for beta-dynamical system
\newblock {\em Adv. Math.}  \textbf{228}  2071--2097

\end{thebibliography}
\end{document}